\documentclass{amsart}

\usepackage{amsthm}
\usepackage{amssymb}
\usepackage[margin=3.2cm]{geometry}

\keywords{Weihrauch reducibility, computable analysis, well-quasiorders, reverse mathematics}

\title{Finding descending sequences through ill-founded linear orders}

\author{Jun Le Goh}
\address{Department of Mathematics\\
University of Wisconsin, Madison}
\email{junle.goh@wisc.edu}

\author{Arno Pauly}
\address{Department of Computer Science\\
University of Swansea}
\email{arno.m.pauly@gmail.com}

\author{Manlio Valenti}
\address{Department of Mathematics, Computer Science and Physics\\
University of Udine}
\email{manlio.valenti@uniud.it}

\thanks{The second author is grateful to Mathieu Hoyrup for answering a question on the impact of the codomain on the deterministic part, and to Linda Westrick for several helpful discussions.
The third author thanks Alberto Marcone for many valuable suggestions during the preparation of the draft. He also thanks  Damir Dzhafarov, Giovanni Sold\`a and Vittorio Cipriani for useful discussions on the topics of the paper. His research was partially supported by the Italian PRIN 2017 Grant ``Mathematical Logic: models, sets, computability".
The authors made significant progress while all attending the BIRS-CMO workshop ``Reverse Mathematics of Combinatorial Principles (19w5111)''.}

\usepackage{ifthen}
\usepackage{color}

\newcommand{\textdef}[1]{\textit{#1}}

\newcommand{\st}{:}
\newcommand{\sequence}[2]{(#1)_{#2}}

\newcommand{\coding}[1]{\langle #1 \rangle}
\newcommand{\defined}{\downarrow}

\newcommand{\defiff}{:\hspace{-1mm}\iff}

\newcommand{\pfunction}{:\subseteq}
\newcommand{\mfunction}[2]{:#1 \rightrightarrows #2}
\newcommand{\pmfunction}[2]{\pfunction #1 \rightrightarrows #2}

\newcommand{\charfun}[1]{\chi_{#1}}
\newcommand{\chiPi}{\charfun{\Pi^1_1} }

\newcommand{\setdifference}{\backslash}
\newcommand{\setcomplement}[2][]{\ifthenelse{\equal{#1}{}}{#2^\mathrm{C}}{ {#1}\setdifference {#2} }}

\newcommand{\ran}{\operatorname{ran}}
\newcommand{\dom}{\operatorname{dom}}

\newcommand{\id}{\operatorname{id}}

\newcommand{\length}[1]{|#1|}
\newcommand{\prefix}{\sqsubseteq}
\newcommand{\pprefix}{\sqsubset}
\newcommand{\concat}{\raisebox{.9ex}{\ensuremath\smallfrown} }

\newcommand{\abslength}[1]{|#1|}

\newcommand{\Baire}{{\mathbb{N}^\mathbb{N}}}
\newcommand{\baire}{{\mathbb{N}^{<\mathbb{N}}}}
\newcommand{\Cantor}{{2^\mathbb{N}}}

\newcommand{\boldfaceDelta}{\boldsymbol{\Delta}}
\newcommand{\boldfaceSigma}{\boldsymbol{\Sigma}}
\newcommand{\boldfacePi}{\boldsymbol{\Pi}}
\newcommand{\boldfaceGamma}{\boldsymbol{\Gamma}}

\newcommand{\repmap}[1]{\delta_{#1}}

\newcommand{\LO}{\mathrm{LO}}
\newcommand{\QO}{\mathrm{QO}}
\newcommand{\WO}{\mathrm{WO}}
\newcommand{\KB}{\mathrm{KB}}

\newcommand{\repTree}{\mathbf{Tr}}

\newcommand{\weireducible}{\le_{\mathrm{W}}}
\newcommand{\strictlyweireducible}{<_\mathrm{W}}
\newcommand{\weiequiv}{\equiv_{\mathrm{W}}}
\newcommand{\weiincomparable}{~|_{\mathrm{W}~}}
\newcommand{\strongweireducible}{\le_{\mathrm{sW}}}

\newcommand{\strongweiequiv}{\equiv_{\mathrm{sW}}}

\newcommand{\muchnikreducible}{\le_w}

\newcommand{\LPO}{ \mathsf{LPO} }
\newcommand{\KL}{ \mathsf{KL} }
\newcommand{\mflim}{\mathsf{lim}}
\newcommand{\mfJ}{\mathsf{J}}
\newcommand{\CNatural}{\mathsf{C}_{\mathbb{N}}}
\newcommand{\UCNatural}{\mathsf{UC}_{\mathbb{N}}}

\newcommand{\RT}[2]{\mathsf{RT}^{{#1}}_{{#2}} }
\newcommand{\cRT}[2]{\mathsf{cRT}^{{#1}}_{{#2}} }

\newcommand{\ATR}{\mathsf{ATR}}
\newcommand{\wList}{\ensuremath{\mathsf{wList}}}

\newcommand{\Comprehension}[1]{{#1}\text{-}\mathsf{CA}}
\newcommand{\PiCA}{\Comprehension{\boldfacePi^1_1}}

\newcommand{\UCBaire}{\mathsf{UC}_\Baire}
\newcommand{\CBaire}{\mathsf{C}_\Baire}
\newcommand{\TCBaire}{\mathsf{TC}_\Baire}
\newcommand{\CCantor}{\mathsf{C}_\Cantor}

\newcommand{\findC}[2]{\mathsf{FindC}^{{#1}}_{{#2}}}
\newcommand{\findS}[1]{\mathsf{FindS}^{#1}}
\newcommand{\ADS}{\mathsf{ADS}}
\newcommand{\GSADS}{\mathsf{General}\text{-}\mathsf{SADS}}

\newcommand{\codedBS}[1]{\ifthenelse{\equal{#1}{}}{\mathsf{BS}}{ {#1}\text{-}\mathsf{BS} }}
\newcommand{\codedDS}[1]{\ifthenelse{\equal{#1}{}}{\mathsf{DS}}{ {#1}\text{-}\mathsf{DS} }}
\newcommand{\DS}{\mathsf{DS}}
\newcommand{\BS}{\codedBS{}}

\newcommand{\codedBound}[1]{{#1}\mathsf{-Bound}}
\newcommand{\PiBound}{\codedBound{\boldfacePi^1_1}}

\newcommand{\codedChoice}[3]{\ifthenelse{\equal{#1}{}}{\mathsf{C}^{\vphantom{g}\mathsf{#2}}_{{#3}} }{ {#1}\text{-}\mathsf{C}^{\vphantom{g}\mathsf{#2}}_{{#3}} }}

\newcommand{\Choice}[1]{\codedChoice{}{}{#1}}

\newcommand{\parallelization}[1]{\widehat{#1}}
\newcommand{\compproduct}{*}

\newcommand{\firstOrderPart}[1]{{}^1{#1}}
\newcommand{\DetPartX }[2]{\mathrm{Det}_{ #2 }(#1)}
\newcommand{\DetPart}[1]{\DetPartX{#1}{}}

\input{thref}

\usepackage[hidelinks]{hyperref}
\usepackage{tikzit}

\tikzstyle{box}=[fill={rgb,255: red,228; green,228; blue,228}, draw=black, shape=rectangle]

\tikzstyle{reducible}=[->, dashed]
\tikzstyle{strictreducible}=[->]
\tikzstyle{nonreducible}=[->, draw=red]
\tikzstyle{uncomp}=[draw=red, <->]

\tikzstyle{thmref}=[opacity=0,inner sep=2mm]

\usepackage{float}

\usepackage[numbers]{natbib}
\usepackage[resetlabels]{multibib}
\newcites{err}{References}

\newtheorem{theorem}{Theorem}[section]
\newtheorem{proposition}[theorem]{Proposition}
\newtheorem{lemma}[theorem]{Lemma}
\newtheorem{corollary}[theorem]{Corollary}

\theoremstyle{definition}
\newtheorem{definition}[theorem]{Definition}

\newtheorem{question}[theorem]{Open Question}
\newtheorem{example}[theorem]{Example}

\subjclass{Primary: 03D30; Secondary: 03D78, 06A75.}

\begin{document}

\begin{abstract}
	In this work we investigate the Weihrauch degree of the problem $\DS$ of finding an infinite descending sequence through a given ill-founded linear order. 
	We show that $\DS$, despite being hard to solve (it has computable inputs with no hyperarithmetic solution), is rather weak in terms of uniform computational strength. To make the latter precise, we introduce the notion of the deterministic part of a Weihrauch degree. We then generalize $\DS$ 
	by considering $\boldfaceGamma$-presented orders, where $\boldfaceGamma$ is a Borel pointclass or $\boldfaceDelta^1_1$, $\boldfaceSigma^1_1$, $\boldfacePi^1_1$. We study the obtained $\DS$-hierarchy 
	of problems in comparison with the (effective) Baire hierarchy and show that 
	it does not collapse at any finite level.
\end{abstract}

\begin{center}
	\bfseries\uppercase{Errata}
	\makeatletter
	\def\@currentlabel{Errata}
	\label{errata}
	\makeatother
\end{center}

\vspace*{1cm}

In the original version of this work, we claimed that the problem $\DS$ (find an infinite descending sequence through an ill-founded linear order) and the problem $\BS$ (find an infinite bad sequence through a non-well quasi-order) are Weihrauch equivalent (\thref{thm:bs=ds}). Recently, Takayuki Kihara found a critical gap in our proof. In fact, $\DS\strictlyweireducible \BS$ \citeerr{GPVweaknessDS}.

The following is a list of the results that are affected:
\begin{itemize}
	\item \thref{thm:bs=ds};
	\item \thref{thm:delta0kDS=delta0kBS} and \thref{thm:delta11DS=delta11BS}: these are one-line relativizations of \thref{thm:bs=ds};
	\item \thref{thm:recap_coded}: the equivalences $\codedDS{\boldfacePi^0_k} \weiequiv \codedDS{\boldfaceDelta^0_{k+1}} \weiequiv \codedDS{\boldfaceSigma^0_{k+1}}$ are unaffected, but the reductions involving $\codedBS{\boldfacePi^0_k}$ and $\codedBS{\boldfaceDelta^0_{k+1}}$ were obtained using \thref{thm:delta0kDS=delta0kBS} and transitivity.
\end{itemize}

At the moment of writing, we are not aware whether all the above-mentioned claims admit a counterexample. 

All the other results are not using \thref{thm:bs=ds}, and either they only deal with $\DS$ or they are standalone results about $\BS$ which are not affected by the above error. In particular,
\thref{thm:codedBS=jump_of_BS}, 
\thref{thm:Sigma0kDS_LQO=Delta0_k+1DS_LO}, \thref{thm:Delta11-DS=DS*UCBaire}, \thref{thm:Pi11ca<Sigma11DS_LQO}, and \thref{thm:Pi11DS<Pi11BS} are correct to the best of our knowledge.

In order to keep the same structure and numbering of the published version, we only applied minimal changes to the following draft. In particular, we rephrased the parts where the equivalence between $\DS$ and $\BS$ was stated in words and the proofs of the theorems affected by the error have been replaced with a reference to this erratum. We also updated Figure~\ref{fig:summary_coded} to reflect our current knowledge on the relation between the coded versions of $\DS$.

\vspace*{1cm}

{
\let\oldaddcontentsline\addcontentsline
\renewcommand{\addcontentsline}[3]{}
\bibliographystyleerr{amsplain}
\bibliographyerr{bibliography}
\let\addcontentsline\oldaddcontentsline
}

\newpage

\maketitle

\vspace*{-3ex}

\tableofcontents
\vspace*{-4ex}

\section{Introduction}

We study the difficulty of the following two 
computational problems:
\begin{itemize}
	\item Given an ill-founded countable linear order, find an infinite decreasing sequence in it ($\DS$) 
	\item Given a countable quasi-order which is not well, find a bad sequence in it ($\BS$).
\end{itemize}
Motivation for the first stems from the treatment of ordinals in reverse mathematics. When working within submodels of second order arithmetic, the notion of well-order depends on the fixed model. This leads to the so-called pseudo-well-orders, i.e.\ ill-founded linear orders s.t.\ no descending sequence exists within the model itself. Such a linear order would appear to be well-founded from the point of view of the model. As a classic example of a pseudo-well-order, consider Kleene's computable linear order with no hyperarithmetic descending sequence (\cite[Lem.\ III.2.1]{SacksHRT}). Such a linear order is a well-order when seen within the $\omega$-model $\mathrm{HYP}$ consisting exactly of the hyperarithmetic sets. Pseudo-well-orders were first studied in \cite{Harrison68} and proved to be a powerful tool in reverse mathematics, especially when working at the level of $\mathrm{ATR}_0$ (see \cite[Sec.\ V.4]{Simpson09}). Our first task can essentially be rephrased as being concerned with the difficulty of revealing a pseudo-ordinal as not actually being an ordinal.

Our second task can be seen as a abstraction of the computational content of theorems in well-quasi-order (wqo) theory. There are many famous theorems asserting that wqo's are closed under certain operations. Examples such as Kruskal's tree theorem, as well as Extended Kruskal's theorem and Higman's theorem, have been well-studied in proof theory via their proof-theoretic ordinals (see \cite{SimpsonWQO}). However, in their usual form these results lack computational content. Indeed, these theorems state that a certain quasi-order $(Q,\preceq_Q)$ is a wqo. Phrasing a result of this kind in the classical $\Pi^1_2$-form would yield a statement of the type ``given an infinite sequence $\sequence{q_n}{n\in\mathbb{N}}$ in $Q$, find a pair of indexes $i<j$ s.t.\ $q_i \preceq_Q q_j$''. Such a pair $(i,j)$ would be a witness of the fact that the sequence $\sequence{q_n}{n\in\mathbb{N}}$ is not bad. However, while proving that $(Q,\preceq_Q)$ is a wqo can be ``hard'' (in particular Extended Kruskal's theorem is not provable in $\boldfacePi^1_1\mathrm{-CA}_0$ \cite{SimpsonWQO}), producing a pair of witnesses for each infinite sequence is a $\preceq_Q$-computable problem (as it can be solved by an extensive search)! 

These theorems are very extreme examples of a well-known difference between reverse mathematics and computable analysis: quoting \cite{GM09} ``the computable analyst is allowed to conduct an unbounded search for an object that is guaranteed to exist by (nonconstructive) mathematical knowledge, whereas the reverse mathematician has the burden of an existence proof with limited means''. 

On the other hand, considering the contrapositives of the above theorems can reveal some (otherwise hidden) computational content. For example, to show that a given quasi-order is not a wqo it suffices to produce a bad sequence in it. Extended Kruskal's theorem or Higman's theorem can be stated in the form ``given a bad sequence for the derived quasi-order, find a bad sequence for the original quasi-order". Our second problem trivially is an upper bound for all these statements, as we disregard any particular reason for why the given quasi-order is not a wqo, and just start with the promise that it is not. Our results thus lay the groundwork for a future exploration of the computational content of individual theorems from wqo theory.

We use the framework of Weihrauch reducibility for our investigation. This means that we compare the problems under investigation to a scaffolding of benchmark problems by asking whether there is an otherwise computable uniform procedure that solves one problem while invoking a single oracle call to the other problem. We are not constrained to particular weak systems in proving that these procedures are actually correct, but rather use whatever proof techniques of ordinary mathematics are suitable. In particular, we can take aspects like the ill-foundedness of the given linear order as external promises not represented in the coding of the input. We can use the fact freely in reasoning about the correctness of our procedure, but there is no evidence provided as input of the procedure.

\subsection{Summary of our results} There are a number of problems whose degrees are milestones in the Weihrauch lattice and are often used as benchmarks to calibrate the uniform strength of the multi-valued function under analysis. Some of them roughly correspond to the so-called \textit{big five} subsystems of second order arithmetic: computable problems correspond to $\mathrm{RCA}_0$, $\CCantor$ (closed choice on the Cantor space) corresponds to $\mathrm{WKL}_0$, $\mflim$ (limit in the Baire space) and its iterations correspond to $\mathrm{ACA}_0$, $\CBaire$ (closed choice on the Baire space) and its variants $\UCBaire$ and $\TCBaire$ correspond to $\mathrm{ATR}_0$, $\PiCA$ corresponds to $\boldfacePi^1_1\mathrm{-CA}_0$.

We show that $\DS$ does not belong to this ``explored'' part of the lattice. To put it in a nutshell, our results show that it is difficult to solve $\DS$, but that $\DS$ is rather weak in solving other problems. For example, $\DS$ has computable inputs without any hyperarithmetic solutions. On the other hand, $\DS$ uniformly computes only the limit computable functions. We provide a few characterizations that tell us what the greatest Weihrauch degree with representatives of particular types below $\DS$ is, and include some general observations on this approach. The diagram in Figure~\ref{fig:summary} shows the relations between $\DS$ and several other Weihrauch degrees. Dashed arrows represent Weihrauch reducibility in the direction of the arrow, solid arrows represent strict Weihrauch reducibility.
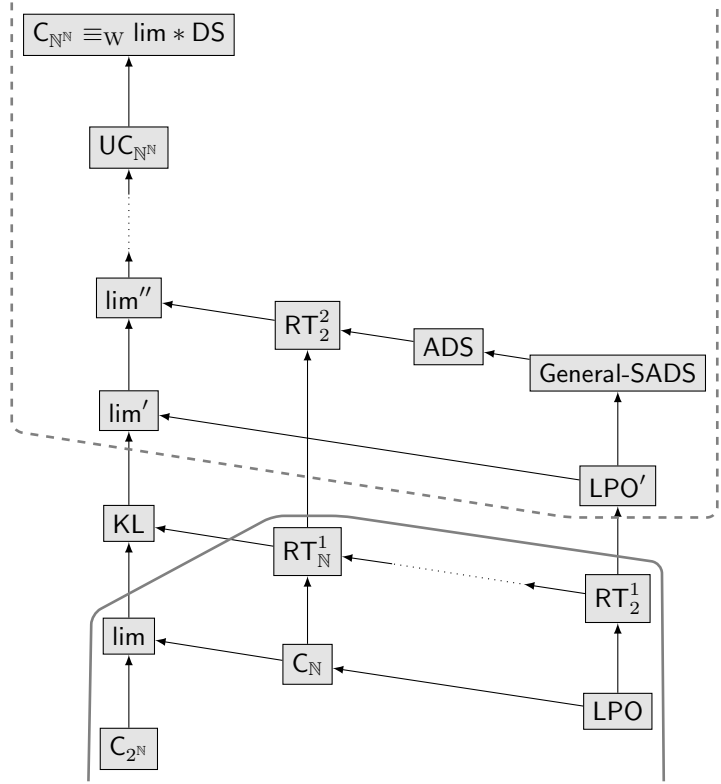
\begin{figure}[tbh]
	\begin{center}
		\tikzstyle{every picture}=[tikzfig]
		\begin{tikzpicture}
	\begin{pgfonlayer}{nodelayer}
		\node [style=box] (1) at (-1, 14) {$\UCBaire$};
		\node [style=box] (3) at (-1, 17) {$\CBaire \weiequiv \mflim\compproduct \DS$};
		\node [style=box] (7) at (-1, 1) {$\mflim$};
		\node [style=box] (9) at (12, 5) {$\LPO'$};
		\node [style=box] (13) at (3.75, 3.25) {$\RT{1}{\mathbb{N}}$};
		\node [style=box] (14) at (-1, 4) {$\mathsf{KL}$};
		\node [style=box] (15) at (12, 8) {$\GSADS$};
		\node [style=box] (16) at (-1, 7) {$\mflim'$};
		\node [style=box] (17) at (-1, 10) {$\mflim''$};
		\node [style=box] (18) at (12, 2) {$\RT{1}{2}$};
		\node [style=box] (19) at (12, -1) {$\LPO$};
		\node [style=box] (20) at (3.75, 0.25) {$\CNatural$};
		\node [style=box] (21) at (-1, -2) {$\CCantor$};
		\node [style=box] (22) at (7.5, 8.7) {$\ADS$};
		\node [style=box] (23) at (3.75, 9.27) {$\RT{2}{2}$};
		\node [style=none] (24) at (-1, 11.25) {};
		\node [style=none] (25) at (-1, 12.75) {};
		\node [style=none] (26) at (6, 2.925) {};
		\node [style=none] (27) at (9.5, 2.38) {};
	\end{pgfonlayer}
	\begin{pgfonlayer}{edgelayer}
		\draw [style=strictreducible] (7) to (14);
		\draw [style=strictreducible] (13) to (14);
		\draw [style=strictreducible] (9) to (15);
		\draw [style=strictreducible] (16) to (17);
		\draw [style=strictreducible] (14) to (16);
		\draw [style=strictreducible] (1) to (3);
		\draw [style=strictreducible] (21) to (7);
		\draw [style=strictreducible] (20) to (7);
		\draw [style=strictreducible] (20) to (13);
		\draw [style=strictreducible] (19) to (20);
		\draw [style=strictreducible] (19) to (18);
		\draw [style=strictreducible] (18) to (9);
		\draw [style=strictreducible] (15) to (22);
		\draw [style=strictreducible] (22) to (23);
		\draw [style=strictreducible] (23) to (17);
		\draw [style=strictreducible] (9) to (16);
		\draw [style=strictreducible] (13) to (23);
		\draw [style=strictreducible] (17) to (24.center);
		\draw [style=strictreducible] (25.center) to (1);
		\draw [style=strictreducible] (18) to (27.center);
		\draw [style=strictreducible] (26.center) to (13);
		\draw [dotted] (24) to (25);
		\draw [dotted] (26) to (27);
		\draw [gray, rounded corners, line width=1pt] plot coordinates {($(21.south west) + (-0.3,-0.3)$) ($(7.north west) + (-0.3,0.1)$) ($ (13.north west)+(0,0.3) $) ($ (13.north east)+(0,0.3) $) ($(18.north east) + (0.3,0.3)$) ($ (21.south -| 19.east) + (0.3,-0.3)$)};
		\draw [gray, dashed, rounded corners, line width=1pt] plot coordinates {($(3.north west)+ (-0.3,0.3)$) ($(3.west |- 16.south)+ (-0.3,0)$) ($(9.south west)+(-0.3,-0.3)$) ($(9.south -| 15.east)+(0.3,-0.3)$) ($(3.north -| 15.east)+(0.3,0.3)$) };
	\end{pgfonlayer}
\end{tikzpicture}
		\caption{An overview of some parts of the Weihrauch lattice. The solid frame collects the degrees belonging to the lower cone of $\DS$, while the dashed frame collects principles that are not Weihrauch reducible to $\DS$. The only principle shown which is above $\DS$ is $\CBaire$. We do not know whether $\KL$ is reducible to $\DS$.}
		\label{fig:summary}
	\end{center}
\end{figure}
Next, we generalize our results by exploring how different presentations of the same order can affect the uniform strength of the same computational task (finding descending sequences in it). We study the problems $\codedDS{\boldfaceGamma}$ and $\codedBS{\boldfaceGamma}$, where the name of the input order carries ``less accessible information'' on the order itself (namely $a \le_L b$ is assumed to be a $\Gamma$-condition relative to the name of the order). We summarize the results in Figure~\ref{fig:summary_coded}.
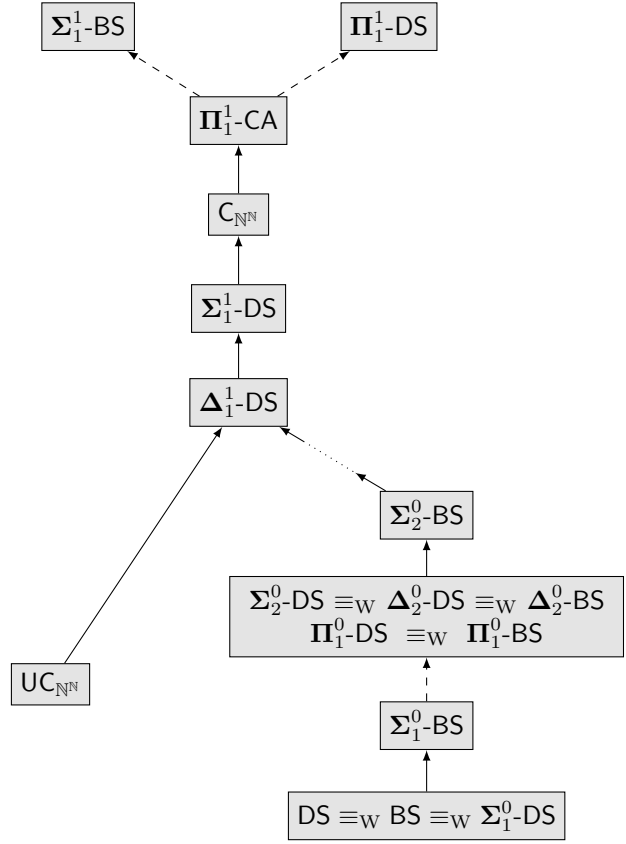
\begin{figure}[tbh]
	\begin{center}
		\tikzstyle{every picture}=[tikzfig]
		\begin{tikzpicture}
	\begin{pgfonlayer}{nodelayer}
		\node [style=box] (1) at (4, 7.5) {$\UCBaire$};
		\node [style=box] (3) at (9, 20) {$\CBaire$};
		\node [style=box] (4) at (14, 5) {$\DS\weiequiv \codedDS{\boldfaceSigma^0_1}$};

		\node [style=box, align=center, text width=5cm] (5) at (14, 8) {
		$\codedDS{\boldfaceSigma^0_2} \weiequiv \codedDS{\boldfaceDelta^0_2} \weiequiv \codedDS{\boldfacePi^0_1}$};
		\node [style=box, align=center, text width=5cm] (12) at (14, 11) {
		$\codedDS{\boldfaceSigma^0_3} \weiequiv \codedDS{\boldfaceDelta^0_3} \weiequiv \codedDS{\boldfacePi^0_2}$};
		\node [style=box] (13) at (9, 15) {$\codedDS{\boldfaceDelta^1_1}$};
		\node [style=box] (14) at (9, 17.5) {$\codedDS{\boldfaceSigma^1_1}$};
		\node [style=box] (15) at (9, 25) {$\codedDS{\boldfacePi^1_1}$};
		\node [style=box] (20) at (9, 22.5) {$\PiCA$};
		\node [style=none] (24) at (12.4, 12.3) {};
		\node [style=none] (25) at (10.6, 13.7) {};
	\end{pgfonlayer}
	\begin{pgfonlayer}{edgelayer}
		\draw [style=strictreducible] (1) to (13);
		\draw [style=strictreducible] (5) to (12);
		\draw [style=strictreducible] (13) to (14);
		\draw [style=strictreducible] (14) to (3);
		\draw [style=strictreducible] (3) to (20);
		\draw [style=reducible] (20) to (15);
		\draw [style=strictreducible] (4) to (5);
		\draw [style=strictreducible] (12) to (24.center);
		\draw [style=strictreducible] (25.center) to (13);
		\draw [dotted] (24) to (25);
	\end{pgfonlayer}
\end{tikzpicture}
		\caption{Diagram presenting the relations between the various generalizations of $\DS$.}
		\label{fig:summary_coded}
	\end{center}
\end{figure}

\subsection{Structure of the paper} After a short introduction on the preliminary notions on represented spaces and Weihrauch reducibility (Section~\ref{sec:background}), we define the deterministic part of a multi-valued function and explore the algebraic properties of the operator $\DetPartX{\cdot}{\mathbf{X}}$ (Section~\ref{sec:detpart}). These results will be very useful in the study of the problems $\DS$ and $\BS$ (Section~\ref{sec:ds}) and their generalizations $\codedDS{\boldfaceGamma}$ and $\codedBS{\boldfaceGamma}$ (Section~\ref{sec:presentations}).

\section{Background}
\label{sec:background}
For an introduction to Weihrauch reducibility, we point the reader to \cite{BGP17}; for represented spaces to \cite{Pauly16}. Below we briefly introduce the notions we will need, as well as state useful results. Those familiar with Weihrauch reducibility should read \thref{def:first_order_part} where we define the first-order part of a problem, recently studied by Dzhafarov, Solomon, Yokoyama \cite{DSYFirstOrder}.

A \emph{represented space} $\textbf{X}$ is a set $X$ together with a (possibly partial) surjection $\repmap{\textbf{X}} \pfunction \Baire \to X$. We can transfer notions of computability from $\Baire$ to $\textbf{X}$ as follows. For each $x \in X$, we say that $p$ is a \textdef{($\repmap{\textbf{X}}$-)name of $x$} if $\repmap{\textbf{X}}(p) = x$. We say that $x \in X$ is \textdef{($\repmap{\mathbf{X}}$-)computable} if it has a computable ($\repmap{\textbf{X}}$-)name. 

We list some relevant examples. Let $\mathbf{LO} = (\LO,\repmap{\mathbf{LO}})$ be the represented space of linear orders with domain contained in $\mathbb{N}$, where each linear order $(L, \le_L)$ is represented by the characteristic function of the set $\{\coding{a,b} \in \mathbb{N} \st a\le_L b\}$. Let $\textbf{WO} = (\WO,\repmap{\mathbf{WO}})$ be the represented space of well-orders with domain contained in $\mathbb{N}$, where $\repmap{\WO}$ is the restriction of $\repmap{\LO}$ to codes of well-orders. Similarly, let $\mathbf{QO} = (\QO,\repmap{\mathbf{QO}})$ be the represented space of quasi-orders (represented via the characteristic function of the relation). Let also $\repTree$ be the space of subtrees of $\baire$, represented by characteristic functions. For every string $\sigma \in \baire$ we denote with $\sigma[n]$ the prefix of length $n$ of $\sigma$.

We will formalize the problems under investigation as partial multi-valued functions between represented spaces $f \pmfunction{\textbf{X}}{\textbf{Y}}$. For each $x \in X$, $f(x)$ denotes the set of possible outputs corresponding to the input $x$. The \emph{domain} $\dom(f)$ is the set of all $x \in X$ such that $f(x)$ is nonempty. We often refer to each $x \in \dom(f)$ as an \emph{$f$-instance} and each $y \in f(x)$ as an \emph{$f$-solution to $x$}. When we define a problem, we will often not specify its domain explicitly, in which case its domain should be taken to be as large as possible. The \emph{codomain} of $f \pmfunction{\textbf{X}}{\textbf{Y}}$ is $\textbf{Y}$. If $f \pmfunction{\textbf{X}}{\textbf{Y}}$ is such that $f(x)$ is a singleton for each $x \in \dom(f)$, then we say that $f$ is \emph{single-valued}. We indicate that by writing $f \pfunction \textbf{X} \to \textbf{Y}$. In this case we will write $f(x) = y$ instead of (the formally correct) $f(x) = \{y\}$. An example of a single-valued problem is the identity function $\id: \Baire \to \Baire$.

We can define the computability or continuity of problems via realizers: we say that a function $F \pfunction \Baire \to \Baire$ is a \emph{realizer} of a problem $f \pmfunction{\textbf{X}}{\textbf{Y}}$ if whenever $p$ is a name for some $x \in \dom(f)$, $F(p)$ is a name for some $y \in f(x)$. A problem is \emph{computable} (respectively \emph{continuous}) if it has a computable (respectively continuous) realizer.

In order to measure the relative uniform computational strength of problems, we use Weihrauch reducibility. A problem $f$ is \emph{Weihrauch reducible} to a problem $g$, written $f \weireducible g$, if there are computable maps $\Phi, \Psi \pfunction \Baire \to \Baire$ such that if $p$ is a name for some $x \in \dom(f)$, then
\begin{enumerate}
	\item $\Phi(p)$ is a name for some $y \in \dom(g)$;
	\item if $q$ is a name for some element of $g(y)$, then $\Psi(p,q)$ is a name for some element of $f(x)$.
\end{enumerate}
This means that there is a procedure for solving $f$ which is computable except for a single invocation to an oracle for $g$. Equivalently, there is a computable procedure which transforms realizers for $g$ into realizers for $f$. A problem $f$ is \emph{strongly Weihrauch reducible} to a problem $g$, written $f \strongweireducible g$, if there are computable maps $\Phi$ and $\Psi$ as above, except that $\Psi$ is not allowed access to $p$ in its computation.

Weihrauch reducibility and strong Weihrauch reducibility are quasi-orders, so they define a degree structure on problems: $f \weiequiv g$ if $f \weireducible g$ and $g \weireducible f$ (likewise for $\strongweireducible$). Both the Weihrauch degrees and the strong Weihrauch degrees form lattices (see \cite[Thm.\ 3.9 and Thm.\ 3.10]{BGP17}). There are several natural operations on problems which also lift to the $\weiequiv$-degrees and the $\strongweiequiv$-degrees. Below we present the operations that we need in this paper. 

The \textdef{parallel product} $f \times g$ is defined by $(f \times g)(x,y) = f(x) \times g(y)$. We call $f$ a \textdef{cylinder} if $f \strongweiequiv f \times \id$. If $f$ is a cylinder, then $g \weireducible f$ if and only if $g \strongweireducible f$ (\cite[Cor.\ 3.6]{BG09}). This is useful for establishing nonreductions because if $f$ is a cylinder, then it suffices to diagonalize against all strong Weihrauch reductions from $g$ to $f$ in order to show that $g \not\weireducible f$. Cylinders will also be useful when working with compositional products (discussed below). Observe that for every problem $f$, $f \times \id$ is a cylinder which is Weihrauch equivalent to $f$.

The \emph{parallelization} $\parallelization{f}$ is defined by $\parallelization{f}((x_n)_{n \in \mathbb{N}}) = \prod_{n \in \mathbb{N}} f(x_n)$. In other words, given a countable sequence of $f$-instances, $\parallelization{f}$ asks for an $f$-solution for each given $f$-instance.

The \emph{composition} of $f \pmfunction{\textbf{Y}}{\textbf{Z}}$ and $g \pmfunction{\textbf{X}}{\textbf{Y}}$ is defined by $\dom(f \circ g) = \{x \in \dom(g) : g(x) \subseteq \dom(f)\}$ and $(f \circ g)(x) = \bigcup_{y \in g(x)} f(y)$ for $x \in \dom(f \circ g)$. The composition does \emph{not} respect $\weireducible$ or $\strongweireducible$. Instead, for any problems $f$ and $g$ (regardless of domain and codomain), we can consider the \emph{compositional product} $f \compproduct g$, which satisfies the following property:
\[ f \compproduct g \weiequiv \max_{\weireducible} \{f_1 \circ g_1 \st f_1 \weireducible f \land g_1 \weireducible g\}. \]
This captures what can be achieved by first applying $g$, possibly followed by some computation, and then applying $f$. The compositional product was first introduced in \cite{BolWei11}, and proven to be well-defined in \cite{BP16}. A useful tool is the \emph{cylindrical decomposition lemma} (\cite[Lem.\ 3.10]{BP16}): for all problems $f$ and $g$, if $F \weiequiv f$ and $G \weiequiv g$ are both cylinders, then there is some computable map $\Phi$ such that $f \compproduct g \weiequiv F \circ \Phi \circ G$. For each problem $f$, let $f^{[n]}$ denote the $n$-fold iteration of the compositional product of $f$ with itself, i.e., $f^{[1]} = f$, $f^{[2]} = f \compproduct f$, and so on.

The \textdef{jump} of $f \pmfunction{\textbf{X}}{\textbf{Y}}$ is the problem $f' \pmfunction{\textbf{X}'}{\textbf{Y}}$ defined by $f'(x) := f(x)$, where $\textbf{X}'$ is the represented space $(X,\repmap{\mathbf{X}'})$ and $\repmap{\mathbf{X}'}$ takes in input a convergent sequence $\sequence{p_n}{n\in\mathbb{N}}$ in $\Baire$ and returns $\repmap{\mathbf{X}}(\lim_{n\to\infty} p_n)$. In other words, $f'$ is the following task: given a sequence which converges to a name of an $f$-instance, produce an $f$-solution to that instance. The jump respects $\strongweireducible$ but does \emph{not} lift to the $\weiequiv$-degrees. We use $f^{(n)}$ to denote the $n$-th jump of a problem. If we define $\mflim \pfunction (\Baire)^{\mathbb{N}} \to \Baire$ by $\mflim((p_n)_{n \in \mathbb{N}}) := \lim_{n \to \infty} p_n$ then it is straightforward from the definition that $f^{(n)}\weireducible f\compproduct \mflim^{[n]}$. The converse reduction does not hold in general. However if $f$ is a cylinder, then for each $n$, $f^{(n)}$ is a cylinder and $f^{(n)} \weiequiv f \compproduct \mflim^{[n]}$ (see \cite[Prop.\ 6.14]{BGP17}). In particular, since  $\mflim$ is a cylinder, $\mflim^{(n)} \weiequiv \mflim^{[n+1]}$ for each $n$. We say that a problem is \emph{arithmetic} if it is Weihrauch reducible to $\mflim^{(n)}$ for some $n$.

Next we introduce some problems which are milestones in the Weihrauch lattice. Apart from $\mflim$ and its jumps, most prominent is the family of problems of \textdef{(closed) choice} $\Choice{\textbf{X}}$ defined below. For a represented space $\mathbf{X}$, let $\mathcal{A}(X)$ denote the space of closed subsets of $\mathbf{X}$. These are given by the ability to recognize membership in the complement. We define $\Choice{\textbf{X}} \pmfunction{\mathcal{A}(X)}{X}$ by $\Choice{\textbf{X}}(A) := A$. In other words, $\Choice{\textbf{X}}$ is the task of producing an element of $\mathbf{X}$ given a way to recognize wrong answers. Define \emph{unique choice} $\mathsf{UC}_{\textbf{X}}$ to be the restriction of $\Choice{\textbf{X}}$ to closed sets which are singletons. By definition, $\mathsf{UC}_{\textbf{X}}$ is single-valued.

Of particular interest to us are $\CNatural$, $\CBaire$ and $\UCBaire$. We can view elements of $\mathcal{A}(\mathbb{N})$ to be given as an enumeration of its complement. Thus, $\CNatural$ is the task of finding a natural number not occurring in a given list. Given a name for a closed set $A \subseteq \Baire$, we can compute a tree $T \subseteq \baire$ such that the set $[T]$ of (infinite) paths on $T$ is $A$. Conversely, given a tree $T \subseteq \baire$, we can compute a name for the closed set $[T]$. Therefore we can view $\CBaire$ as the problem of computing a path on a given ill-founded subtree of $\baire$. Similarly, we can view $\UCBaire$ as the problem of computing the unique path on a given subtree of $\baire$. Both $\CBaire$ and $\UCBaire$ are closed under compositional product \cite[Thm.\ 7.3]{BdBPLow12}. We have:

\begin{theorem}[{\cite[Cor.\ 3.4]{KMP20}}]
	\thlabel{thm:ucbaire_hyperarithmetic_solution}
If $f\pmfunction{\Baire}{\mathbf{X}}$ is Weihrauch reducible to $\UCBaire$, then for every $x\in\dom(f)$, $f(x)$ contains some $y$ hyperarithmetical relative to $x$.
\end{theorem}

Another prominent problem is $\LPO: \Baire \to \{0,1\}$, defined by $\LPO(p) := 0$ if $p = 0^\mathbb{N}$ and $\LPO(p) := 1$ otherwise. Its jump $\LPO'$ (and its iterated jumps $\LPO^{(k)}$) will play an important technical role. We notice that $\mflim^{(n)}\weiequiv\parallelization{\LPO^{(n)}}$ (see e.g.\ \cite[Thm.\ 6.7 and Prop.\ 6.10]{BGP17}).

Next we define the represented space $\boldfaceGamma(\textbf{X})$ of all $\boldfaceGamma$-definable subsets of a computable metric space $X$, where $\boldfaceGamma \in \{ \boldfaceSigma^0_k, \boldfacePi^0_k, \boldfaceDelta^0_k, \boldfaceSigma^1_1, \boldfacePi^1_1, \boldfaceDelta^1_1 \}$. This is based on the well-known concept of Borel codes \cite{moschovakisb}. For a more detailed development in the context of computable analysis we refer to \cite{pauly-gregoriades}. A more abstract and general treatment is provided in \cite{paulydebrecht2-lics}. A $\repmap{\boldfaceSigma^0_1}$-name for a set $B \subseteq X$ is a sequence of indices of rational open balls whose union is $B$. A $\repmap{\boldfacePi^0_k}$-name for a set is a $\repmap{\boldfaceSigma^0_k}$-name for its complement. (Note that $\repmap{\boldfacePi^0_1}$ agrees with how we represented closed sets previously.) A $\repmap{\boldfaceDelta^0_k}$-name for a set is a pair of $\repmap{\boldfaceSigma^0_k}$-names, one for the set itself and one for its complement. A $\repmap{\boldfaceSigma^0_{k+1}}$-name for a set $B \subseteq X$ is a sequence of names for $\boldfacePi^0_k$ sets whose union is $B$.

A $\repmap{\boldfaceSigma^1_1}$-name for a set $S \subseteq X$ is a $\repmap{\boldfacePi^0_1}$-name for a set $P \subseteq \Baire \times X$ such that $S = \{x \in X: (\exists g)((g,x) \in P)\}$. We define $\repmap{\boldfacePi^1_1}$ and $\repmap{\boldfaceDelta^1_1}$ similarly to $\repmap{\boldfacePi^0_k}$ and $\repmap{\boldfaceDelta^0_k}$. If $X$ is $\mathbb{N}$, we can think of a $\repmap{\boldfaceSigma^1_1}$-name for $S \subseteq \mathbb{N}$ as a sequence $(T_n)_{n \in \mathbb{N}}$ of subtrees of $\baire$ such that $n \in S$ if and only if $T_n$ is ill-founded.

In practice, we rarely construct $\repmap{\boldfaceGamma}$-names explicitly. If we want to construct a $\repmap{\boldfaceGamma}$-name for a set $A \subseteq X$, we typically only check that there is a $\boldfaceGamma$-formula which defines $A$. By invoking computable closure properties, one can construct a computable map which takes a $\boldfaceGamma$-formula $\phi$ and its parameter $p$ to a $\repmap{\boldfaceGamma}$-name for the set defined by $\phi$. Conversely, one can construct a computable map which takes a $\repmap{\boldfaceGamma}$-name $p$ for a set $A$ to a $\boldfaceGamma$-formula $\phi$ with parameter $p$ which defines $A$.

We define the (single-valued) functions $\Comprehension{\boldfaceGamma}\pfunction \boldfaceGamma(\mathbb{N})\to \Cantor$ corresponding to comprehension principles: given a $\repmap{\boldfaceGamma}$-name $p$ for a subset $A$ of $\mathbb{N}$, produce its characteristic function. Notice that, for each $k$ and each $A\in \boldfaceSigma^0_{k+1}$, we can use $\LPO^{(k)}$ to check whether $n\in A$ (intuitively, for every $p$ we can use $\LPO^{(k)}$ to answer a $\Sigma^{0,p}_{k+1}$ question). This shows that, for each $k$,
\[ \mflim^{(k)} \weiequiv \parallelization{\LPO^{(k)}} \weiequiv \Comprehension{\boldfaceSigma^0_{k+1}}, \]
somewhat implicitly written in \cite{brattka}. The problem $\PiCA$ can be seen as the analogue of $\boldfacePi^1_1\mathrm{-CA}_0$. It is Weihrauch equivalent to the parallelization of $\chiPi$, which is the characteristic function of a $\Pi^1_1$-complete set. It is convenient to think of $\chiPi$ as the function that takes in input a subtree of $\baire$ and checks whether it is well-founded. 

We can also define $\boldfaceGamma$-choice $\codedChoice{\boldfaceGamma}{}{\textbf{X}} \pmfunction{\boldfaceGamma(\textbf{X})}{\textbf{X}}$ by $\codedChoice{\boldfaceGamma}{}{\textbf{X}}(A) := A$. In other words, $\codedChoice{\boldfaceGamma}{}{\textbf{X}}$ is the task of producing an element of a nonempty $\boldfaceGamma$-set $A$, given a $\repmap{\boldfaceGamma}$-name of $A$. We can define $\boldfaceGamma\text{-}\mathsf{UC}_{\textbf{X}} \pfunction \boldfaceGamma(\textbf{X}) \to \textbf{X}$ analogously. When reducing problems to $\CBaire$ and $\UCBaire$, the following facts are helpful: $\codedChoice{\boldfaceSigma^1_1}{}{\Baire} \weiequiv \CBaire$ and $\UCBaire \weiequiv \boldfaceSigma^1_1\text{-}\mathsf{UC}_{\Baire}$ (see \cite{KMP20}).

We define the represented spaces $\boldfaceGamma(\textbf{LO})$ and $\boldfaceGamma(\textbf{QO})$ by restricting the codomain of $\repmap{\boldfaceGamma}$ to the set of subsets of $\mathbb{N}$ which are characteristic functions of linear orders and quasiorders respectively. This will be used in Section~\ref{sec:presentations}.

We will often construct linear orders using the following method. For every tree $T\subseteq \baire$, we define the \emph{Kleene-Brouwer order} $\KB(T)$ on $T$ as follows: $\sigma \le_{\KB} \tau$ if and only if $\tau \prefix\sigma$ or $\sigma\le_{lex} \tau$. The map $T\mapsto \KB(T)$ from $\repTree$ to $\textbf{LO}$ is computable. It is known that $\KB(T)$ is a well-order if and only if $T$ is well-founded (see e.g.\ \cite[Lem.\ V.1.3]{Simpson09}).

Finally we present a notion recently studied by Dzhafarov, Solomon, Yokoyama \cite{DSYFirstOrder}:

\begin{definition}
    \thlabel{def:first_order_part}
	Let $\mathcal{F}$ be the set of \textdef{first-order} problems, i.e.\ the set of problems with codomain $\mathbb{N}$. For every problem $f\pmfunction{\mathbf{Y}}{\mathbf{Z}}$, the \textdef{first-order part of $f$} is the multi-valued function $\firstOrderPart{f}\pmfunction{\Baire\times \mathbf{Y}}{\mathbb{N}}$ defined as follows:
	\begin{itemize}
		\item instances are pairs $(p,y)$ s.t.\ $y\in \dom(f)$ and for every $z\in f(y)$ and every name $t$ for $z$, $\Phi_p(t)(0)\downarrow$, where $\Phi_{(\cdot)}$ is a fixed universal Turing functional;
		\item a solution for $(p,y)$ is any $n$ s.t.\ there is a name $t$ for a solution $z\in f(y)$ s.t.\ $\Phi_p(t)(0)\downarrow=n$.
	\end{itemize}
\end{definition}

The motivation for this notion comes from the following fact:

\begin{proposition}[{\cite{DSYFirstOrder}}]
	For every problem $f$, 
	\[ \firstOrderPart{f} \weiequiv \max_{\weireducible}\{ g\in \mathcal{F} \st g \weireducible f\}. \]
\end{proposition}

We conclude this section with the following proposition:

\begin{proposition}
	\thlabel{thm:fo_part_cbaire}
	$\firstOrderPart{\CBaire}\weiequiv \codedChoice{\boldfaceSigma^1_1}{}{\mathbb{N}}$.
\end{proposition}
\begin{proof}
	It is known that $\codedChoice{\boldfaceSigma^1_1}{}{\mathbb{N}}\strictlyweireducible \parallelization{\codedChoice{\boldfaceSigma^1_1}{}{\mathbb{N}}}\strictlyweireducible \CBaire$ (\cite[Thm.\ 3.34]{KiharaADauriacChoice}). On the other hand, if $f\pmfunction{\mathbf{X}}{\mathbb{N}}$ is s.t.\ $f\weireducible \CBaire$ via $\Phi,\Psi$ then, for every name $p$ of some $x\in\dom(f)$, $\Phi(p)$ is the name of an ill-founded tree $T_p$ and, for every $t\in [T_p]$ we have $\Psi(t)(0)\in f(x)$. This means that we can compute a solution choosing an element from 
	\[ \{ n \in\mathbb{N} \st (\exists t \in \Baire)(t\in [T_p] \land \Psi(t)(0)=n) \}, \]
	which is a $\Sigma^{1,p}_1$ subset of $\mathbb{N}$.
\end{proof}

We will characterize the first-order part of $\DS$ in 	\thref{thm:below_BS_below_Pi11Bound}.

\section{The deterministic part of a problem}
\label{sec:detpart}

\begin{definition}
	Let $\mathbf{X}$ be a represented space and $f\pmfunction{\mathbf{Y}}{\mathbf{Z}}$ be a multi-valued function. We define $\DetPartX{f}{\mathbf{X}}\pfunction \Baire \times \mathbf{Y} \to \mathbf{X}$ by 
	\[ \DetPartX{f}{\mathbf{X}}(p,y) = x \defiff (\forall z \in \delta_\mathbf{Z}^{-1}(f(y)))(\delta_\mathbf{X}(\Phi_p(z)) = x),\] 
	where $\Phi_{(\cdot)}$ is a universal Turing functional. The domain of $\DetPartX{f}{\mathbf{X}}$ is maximal for this to be well-defined. We just write $\DetPart{f}$ for $\DetPartX{f}{\Baire}$.
\end{definition}
	
Notice that $\DetPart{f}$ is always a cylinder. This is not true for all $\textbf{X}$  (if $\mathbf{X}=\mathbb{N}$ then $\DetPartX{f}{\mathbf{X}}$ always has computable solutions, and therefore $\id\not\strongweireducible \DetPartX{f}{\mathbf{X}}$).

Our interest in the principle $\DetPartX{f}{\mathbf{X}}$ lies in the fact that it has the maximal Weihrauch degree of all (single-valued!) functions with codomain $\mathbf{X}$ that are Weihrauch below $f$:

\begin{theorem}
	\thlabel{thm:detpart_total}
$\DetPartX{f}{\mathbf{X}} \weiequiv \max_{\weireducible} \{g : \subseteq \mathbf{W} \to \mathbf{X} \st g \weireducible f\}$.
\begin{proof}
Clearly, $\DetPartX{f}{\mathbf{X}}$ is itself present in the set on the right hand side. Assume $g\pfunction \mathbf{W} \to \mathbf{X}$ satisfies $g \weireducible f$ with reduction witnesses $\Phi$ and $\Psi$. Given a name $q$ for an input to $g$, let $y = \delta_\mathbf{Y}(\Phi(q))$ be the value $f$ is called on, and let $p$ be a name for the function $\Psi(q,\cdot)$. Then $\DetPartX{f}{\mathbf{X}}(p,y) = g(\delta_\mathbf{W}(q))$.
\end{proof}
\end{theorem}

In the same spirit, we can identify several other operators $\Lambda_\mathcal{Y}$ of the type $\Lambda_\mathcal{Y}(f):= \max_{\weireducible} \{ g \in \mathcal{Y} \st g\weireducible f \}$. In particular the proof strategy used in \thref{thm:detpart_total} can be used to prove that $\Lambda_{\mathcal{U}_N}$ and $\Lambda_{\mathcal{V}_N}$ are total, where $\mathcal{U}_N$ is the set of first-order problems with codomain $N$, and $\mathcal{V}_N$ is the set of problems in $\mathcal{U}_N$ which are also single-valued. This will come into play in \thref{thm:below_DS_below_RT1k} and in \thref{thm:below_DS_below_lim_k}.

\begin{corollary}
	$\DetPartX{\cdot}{\mathbf{X}}$ is an interior degree-theoretic operator on Weihrauch degrees, i.e.\ 
	\begin{gather*}
		\DetPartX{\DetPartX{f}{\mathbf{X}}}{\mathbf{X}} \weiequiv \DetPartX{f}{\mathbf{X}} \weireducible f; \\
		f \weireducible g \Rightarrow \DetPartX{f}{\mathbf{X}} \weireducible \DetPartX{g}{\mathbf{X}}.
	\end{gather*}
\end{corollary}

\subsection{Impact of the codomain space}
We make some basic observations on how the space $\mathbf{X}$ impacts the degrees $\DetPartX{f}{\mathbf{X}}$ for arbitrary $f$. Clearly, whenever $\mathbf{Y}$ computably embeds into $\mathbf{X}$ (i.e.\ there is a computable injection $\mathbf{Y}\to\mathbf{X}$ with computable inverse), then $\DetPartX{f}{\mathbf{Y}} \weireducible \DetPartX{f}{\mathbf{X}}$. In general, we obtain many different operations. To see this, we consider the point degree spectrum of a represented space as introduced by Kihara and P. \cite{pauly-kihara-arxiv}. The point degree spectrum of $(X,\delta_\mathbf{X})$ is the set of Medvedev degrees of the form $\delta^{-1}_{\mathbf{X}}(x)$ for $x \in X$. 

The spectrum of $\mathbf{Y}$ is included in that of $\mathbf{X}$ iff $\mathbf{Y}$ can be decomposed into countably many parts each of which embeds into $\mathbf{X}$ (\cite[Lem.\ 3.6]{pauly-kihara-arxiv}). If the spectrum of $\mathbf{Y}$ is not included in that of $\mathbf{X}$, we can consider a constant function $y$ witnessing this. Then $\DetPartX{y}{\mathbf{X}} \strictlyweireducible \DetPartX{y}{\mathbf{Y}} \weiequiv y$. We have thus seen that if $\DetPartX{f}{\mathbf{X}} \weiequiv \DetPartX{f}{\mathbf{Y}}$ for all $f$, then $\mathbf{X}$ and $\mathbf{Y}$ must have the same point degree spectrum. Miller \cite{miller2} has shown that the spectrum of $[0,1]^\omega$ is not contained in the Turing degrees (i.e.~the spectrum of $\Cantor$), which was extended in \cite{pauly-kihara-arxiv} to the result that the spectrum of a computable Polish space is contained in the Turing degrees relative to some oracle iff that space is countably dimensional. The spectra of further spaces have been explored in \cite{pauly-kihara3}. 

We can extend the separation arguments based on the spectrum by considering sequences rather than just constant functions\footnote{The ideas in this paragraph were pointed out to us by Mathieu Hoyrup.}. Whenever we have a sequence $f_0 : \mathbb{N} \to \mathbf{X}_0$ and a function $g_0 : \subseteq \Baire \to \mathbf{X}_1$ with $f_0 \weiequiv g_0$, then there is a sequence $h: \mathbb{N} \to \mathbf{X}_1$ with $f_0 \weiequiv h$. A Weihrauch reduction $f \weireducible g$ for $f : \mathbb{N} \to \mathbf{X}$ and $g : \mathbb{N} \to \mathbf{Y}$ gives rise to a computable partial function $F : \subseteq \mathbf{Y}^\mathbb{N} \to \mathbf{X}^\mathbb{N}$ with $F(g) = f$. It follows that it suffices to separate $\mathbf{Y}^\mathbb{N}$ and $\mathbf{X}^\mathbb{N}$ via their spectrum to conclude that $\DetPartX{\cdot}{\mathbf{X}}$ and $\DetPartX{\cdot}{\mathbf{Y}}$ are distinct operators. In particular, Miller's result implies that there is a function with codomain $\mathbb{R}$ that is not equivalent to any function with codomain $\Baire$.

\subsection{The deterministic part and the first-order part}

Let us now explore the interplay between the deterministic part and the first-order part.

\begin{proposition}
	$\firstOrderPart{\DetPart{f}} \weiequiv \DetPartX{f}{\mathbb{N}} \weireducible \DetPart{\firstOrderPart{f}}$.
	\begin{proof}
		By considering what the relevant maxima in the characterizations are taken about, it is clear that $\DetPartX{f}{\mathbb{N}} \weireducible \firstOrderPart{\DetPart{f}}$ and $\DetPartX{f}{\mathbb{N}} \weireducible \DetPart{\firstOrderPart{f}}$. To see that $\firstOrderPart{\DetPart{f}} \weireducible \DetPartX{f}{\mathbb{N}}$, we consider a function $f : \subseteq \Baire \to \Baire$ and a multi-valued function $g\pmfunction{\Baire}{\mathbb{N}}$ with $g \weireducible f$. But this reduction actually yields some choice function of $g$, showing that $g \weireducible \DetPartX{f}{\mathbb{N}}$.
	\end{proof}
\end{proposition}

\begin{question}
	Is there some $f$ with $\DetPartX{f}{\mathbb{N}} \strictlyweireducible \DetPart{\firstOrderPart{f}}$?
\end{question}

The question above asks whether whenever there is a countable cover making a partial function on Baire space piecewise computable, there also is a partition of the same or lower complexity that renders the function piecewise computable. The complexity here is not merely the complexity of the individual pieces, but the Weihrauch degree of the map that assigns the piece to any Baire space element.

\begin{proposition}
	$\DetPart{f} \weireducible \widehat{\DetPartX{f}{\mathbb{N}}}$.
	\begin{proof}
		A function $f : \subseteq \Baire \to \Baire$ is reducible to the parallelization of its uncurried form $F : \subseteq \mathbb{N} \times \Baire \to \mathbb{N}$ where $F(n,p) = f(p)(n)$.
	\end{proof}
\end{proposition}

\begin{corollary}
	$\DetPart{f} \weireducible \widehat{\firstOrderPart{f}}$.
\end{corollary}

\subsection{Interaction with other operations on Weihrauch degrees}
A first straightforward observation is that $\DetPart{f}\,\square\, \DetPart{g} \weireducible \DetPart{f\,\square\, g}$ whenever $\square$ is a degree-theoretic operator that preserves single-valuedness. We will look at the interaction with the usual well-studied operations on Weihrauch degrees. Besides those introduced in Section \ref{sec:background}, we consider $\sqcup$ and $\sqcap$, the join and meet in the Weihrauch lattice, and the finite parallelization $^*$ (which essentially is closure under $\times$). See \cite[Section 3]{BGP17} for definitions. The diamond operator ${}^\diamond$ was introduced in \cite{NP18}, and corresponds to the possibility of using the oracle an arbitrary (but finite) number of times (essentially closure under compositional product).

It is imminent from the definition that $\DetPart{f}\sqcup \DetPart{g} \weiequiv \DetPart{f \sqcup g}$. 

Moreover $\DetPart{f \sqcap g} \weireducible \DetPart{f}$ and $\DetPart{f \sqcap g} \weireducible \DetPart{g}$ by monotonicity, hence $\DetPart{f \sqcap g} \weireducible \DetPart{f}\sqcap \DetPart{g}$, as $\sqcap$ is the meet on Weihrauch degrees (\cite[Prop.\ 3.11]{BG09}). To see that the inequality can be strict, let $p,q \in \Cantor$ be a minimal pair of Turing degrees (which we identify with the constant functions returning these values). Then $\DetPart{p \sqcap q} \weiequiv \id_\Baire \strictlyweireducible \DetPart{p} \sqcap \DetPart{q} \weiequiv p \sqcap q$. 

Our principle $\DS$ (to be defined) already witnesses that the deterministic part does not distribute over $\times$ and $\compproduct$, and does not commute with $^*$, $^\diamond$ and $\widehat{\phantom{f}}$: we will prove that $\DetPart{\DS}\weiequiv \mflim$ (\thref{thm:below_DS_below_lim}), while $\LPO'\weireducible \DS \times \DS$ (\thref{thm:DS_not_closed_under_product}). Here we also give another example with a more computability-theoretic flavour:

\begin{example}
	There is a Weihrauch degree $f$ such that: \[\DetPart{f} \weiequiv \id \strictlyweireducible f \strictlyweireducible f \times f \weiequiv f^\diamond \weiequiv \widehat{f} \weiequiv \DetPart{f \times f}.\]	
	Indeed, consider the degrees of points in the spaces $\mathbb{R}_<$, $\mathbb{R}_>$ and $\mathbb{R}$ (see \cite{pauly-kihara3} for details). Let $x \in \mathbb{R}$ be neither left-c.e.\ nor right-c.e.; i.e.\ it lacks computable names in both $\mathbb{R}_<$ and $\mathbb{R}_>$. Then $x \in \mathbb{R}_<$ and $x \in \mathbb{R}_>$ have quasi-minimal degrees, that is do not compute any non-computable elements of Cantor space. We define $f\colon\mathbf{2} \to \mathbb{R}_< + \mathbb{R}_>$ by $f(0) = x \in \mathbb{R}_<$ and $f(1) = x \in \mathbb{R}_>$. The quasi-minimality implies that $\DetPart{f} \weiequiv \id$. However, $f \times f$ is equivalent to the constant function returning $x \in \mathbb{R}$, which is also equivalent to the constant function returning the decimal expansion of $x$. Thus, $f \times f \weiequiv \DetPart{f \times f}$. Any of $f^*$, $f \compproduct f$, $f^\diamond$ and $\widehat{f}$ clearly share the same degree.
\end{example}

\begin{theorem}
	\thlabel{thm:det(f*g)<=det(f)*g}
	For every represented space $\mathbf{X}$ and every problems $f,g$, 
	\[ \DetPartX{f\compproduct g}{\mathbf{X}} \weireducible \DetPartX{f}{\mathbf{X}} \compproduct g. \] 
\end{theorem}
\begin{proof}
	Fix a single-valued $h$ with codomain $\mathbf{X}$ and assume, without loss of generality, that $\dom(h)\subset \Baire$ (if $h$ is single-valued then the map $p\mapsto h\circ\repmap{}(p)$ is single-valued as well, where $\repmap{}$ is the representation map for the domain of $h$). Assume also, for the sake of readability, that $f$ and $g$ are cylinders (if not we can just replace $f$ with $f \times \id{}$, as $\DetPartX{\cdot}{\mathbf{X}}$ is a degree-theoretic operation).
	
	By the cylindrical decomposition lemma, there is a computable function $\Phi_e$ s.t.\ 
	\[ h \strongweireducible f \circ \Phi_e\circ g. \]
	Let $\Phi,\Psi$ be two maps witnessing this strong reduction. Define $\phi$ as the restriction of $\repmap{\mathbf{X}}\circ\Psi\circ f \circ \Phi_e$ to $\dom(g \circ \Phi \circ h)$. The choice of the domain of $\phi$ guarantees that $\phi$ is single-valued: intuitively $\phi$ witnesses the ``second part'' of the reduction $h\strongweireducible f \circ \Phi_e\circ g$, and the fact that $h$ is single-valued implies that so is $\phi$. In particular, $\phi \weireducible \DetPartX{f}{\mathbf{X}}$ (as $\phi \weireducible f$ trivially). Since $h\weireducible \phi\compproduct g$ we have that $h\weireducible \DetPartX{f}{\mathbf{X}}\compproduct g$.
\end{proof}

Notice that this implies the choice elimination theorem \cite[Thm.\ 7.25]{BGP17}, as $\DetPart{\CCantor}\weiequiv\id{}$ (\cite[Cor.\ 8.8]{BG09}).

\begin{corollary}
	If $g$ is single-valued with codomain $\Baire$ then $\DetPart{f\compproduct g}\weiequiv \DetPart{f} \compproduct g$. 
\end{corollary}
\begin{proof}
	This follows from \thref{thm:det(f*g)<=det(f)*g}, as $\DetPart{f}\compproduct \DetPart{g}\weireducible\DetPart{f\compproduct g}$ always holds and $\DetPart{g}\weiequiv g$ as $g$ is single-valued with codomain $\Baire$.
\end{proof}

\begin{corollary}
	\thlabel{thm:detpart_commutes_jump}
	For every cylinder $f$ and every $k\in\mathbb{N}$
	\[ \DetPart{f}^{(k)} \weiequiv\DetPart{f^{(k)}}. \]
\end{corollary}
\begin{proof}
	The left-to-right reduction is straightforward as
	\[ \DetPart{f}^{(k)} \weireducible \DetPart{f}\compproduct\mflim^{[k]} \weireducible f \compproduct\mflim^{[k]} \weiequiv f^{(k)},\]
	where the last equality follows from the fact that $f$ is a cylinder. Since $\DetPart{f}^{(k)}$ is single-valued, this implies $\DetPart{f}^{(k)}\weireducible \DetPart{f^{(k)}}$.

	The right-to-left reduction follows from \thref{thm:det(f*g)<=det(f)*g} as 
	\[ \DetPart{f^{(k)}} \weiequiv \DetPart{f\compproduct\mflim^{[k]}}\weireducible\DetPart{f} \compproduct\mflim^{[k]} \weiequiv \DetPart{f}^{(k)},\]
	where the last equality follows from the fact that $\DetPart{f}$ is a cylinder.
\end{proof}

The previous corollary can be generalized in a straightforward way to any represented space $\mathbf{X}$ s.t.\ $\DetPartX{f}{\mathbf{X}}$ is a cylinder. Notice that it is false (in general) if $f$ is not a cylinder: take $f=\Choice{2}$ and $k=1$. Since $\mathsf{C}_2'\weiequiv \RT{1}{2}$ (see e.g.\ \cite[Fact 2.3 and Prop.\ 3.4]{BRramsey17}) we have $\DetPart{\mathsf{C}_2'} \weireducible \RT{1}{2}$, 
hence in particular $\mflim \not \weireducible \DetPart{\mathsf{C}_2'}$. On the other hand $\mflim \weireducible \DetPart{\Choice{2}}'$ (as $\DetPart{\Choice{2}}$ is a cylinder).

\begin{definition}
	Given some $f\pmfunction{\Baire}{\Baire}$ let $?f\pmfunction{\Baire}{\Baire}$ be defined by $0^\omega \in ?f(0^\omega)$ and $0^n1p \in ?f(0^n1q)$ iff $p \in f(q)$.
\end{definition}

It is easy to see that $?$ defines an operation on Weihrauch degrees, and represents the idea of being able to maybe ask a question to $f$ -- but never having to decide to forgo this (which would be the case for $1 \sqcup f$). Many well-studied principles are equivalent to their maybe-variants, this in particular holds for all pointed fractals. We introduce the operation here to be able to express how the deterministic part interacts with the notion of completion $\overline{(\cdot)}$ introduced by Brattka and Gherardi \cite{BGBrouwerian,BGCompOfChoice19}.

\begin{proposition}
	$\DetPart{\overline{f}} \weiequiv \DetPart{?f}\weiequiv ?\DetPart{f}$.
	\begin{proof}
		To show that $\DetPart{\overline{f}} \weireducible \DetPart{?f}$, wlog assume that $f\pmfunction{\Baire}{\Baire}$ and consider a function $g : \subseteq \Baire \to \Baire$ with $g \weireducible \overline{f}$ witnessed by $\Phi,\Psi$. Now if for some prefix $w$ the computation of $\Psi(w,\cdot)$ outputs two different things depending on the second part of the input, then in order for $g$ to be a function, we have the guarantee that all extensions of $w$ in the domain of $g$ will be mapped to inputs in the domain of $f$, i.e.\ we are actually calling $f$ rather than making use of $\overline{f}$. On the other hand, if $\Psi(w,\cdot)$ would output the same thing regardless of the second argument, we can postpone actually calling $f$ (which $?f$ lets us do) and go with that output for the time being. This reasoning establishes that $g \weireducible ?f$.
		
		To see that $\DetPart{?f} \weireducible ?\DetPart{f}$, we just inspect the technical definition of $\DetPart{\cdot}$.
		
		Finally, for $?\DetPart{f} \weireducible \DetPart{\overline{f}}$ we observe that $?\DetPart{f}$ is single-valued with codomain $\Baire$, thus it suffices to show $?\DetPart{f} \weireducible \overline{f}$. But already $?f \weireducible \overline{f}$ holds: $\overline{f}$ accepts an input that is completely void of information. We provide this as long as our $?f$ instance does not want to use $f$; if it ever does, we have the relevant $f$-instance which we can then feed into $\overline{f}$. Note that we do not get a strong reduction here, in general.
	\end{proof}
\end{proposition}

\subsection{Previous appearances in the literature}
\label{sec:detpar_literature}
While the deterministic part as such has not been introduced before, and in particular the observation that it is always well-defined is new, there are several results in the literature on Weihrauch degrees that implicitly use it. Already in the first paper introducing the modern definition of Weihrauch reducibility \cite{GM09}, it was shown that $\DetPart{\CCantor} \weiequiv \id$. It was observed in \cite{paulyleroux} that the argument actually even establishes that $\DetPartX{\CCantor}{\mathbf{X}} \weiequiv \id$ for any computably admissible space $\mathbf{X}$.

In \cite{KMP20} the principle $\wList_{\Cantor,\leq\omega}$ which produces an enumeration of the elements of a countable closed subset of Cantor space was introduced, and \cite[Prop.\ 6.14]{KMP20} states that $\DetPart{\wList_{\Cantor,\leq\omega}} \weiequiv \mflim$. The authors also proved the following result, which will be useful in \thref{thm:detpart_Delta11-DS}:
\begin{theorem}[{\cite[Thm.\ 8.5]{KMP20}}]
	\thlabel{thm:detpart_cbaire}
	$\UCBaire \weiequiv \DetPart{\CBaire}\weiequiv\DetPart{\parallelization{\TCBaire}}$.
\end{theorem}
This, in particular, shows that $\DetPart{\cdot}$ is not useful to separate principles that are between $\UCBaire$ and $\CBaire$.

In the context of probabilistic computation \cite{hoelzl,hoelzl2}, the fact that the upper cones of non-trivial enumeration degrees are measure zero is equivalent to the statement that if $f\pmfunction{\Baire}{\Baire}$ has the property that $f(p)$ has positive measure for every $p \in \dom(f)$ and $\mathbf{X}$ is an effectively countably based space, then there is some $g$ with $\DetPartX{f}{\mathbf{X}} \weireducible \firstOrderPart{g}$.

\section{Finding descending sequences}
\label{sec:ds}
Let us formally define the problem of finding descending sequences in an ill-founded linear order as a multi-valued function.
\begin{definition}
	Let $\DS\pmfunction{\mathbf{LO}}{\Baire}$ be the multi-valued function defined as 
	\[ \DS(L):= \{ x\in\Baire \st (\forall i)(x(i+1)<_L x(i)) \}, \] 
	with $\dom(\DS):= \LO \setdifference \WO$.
\end{definition}

\subsection{The uniform strength of \texorpdfstring{$\DS$}{DS}}

We can immediately notice the following:

\begin{proposition}
	\thlabel{thm:bs_not_<=_ucbaire}
	$\DS\weireducible \CBaire$ but $\DS\not\weireducible \UCBaire$.
\end{proposition}
\begin{proof}
	To show that $\DS\weireducible\CBaire$ it is enough to notice that being a descending sequence in a linear order $L$ is a $\Pi^{0,L}_1$ property. In other words, we can obtain a descending sequence through $L$ by choosing a path through the tree 
	\[ \{ \sigma\in \baire \st (\forall i<\length{\sigma}-1)(\sigma(i+1)<_L \sigma(i)) \}. \]

	To show that $\DS\not\weireducible\UCBaire$, recall that there is a computable linear order with no hyperarithmetic descending sequence (see e.g.\ \cite[Lem.\ III.2.1]{SacksHRT}). A reduction $\DS\weireducible\UCBaire$ would therefore contradict \thref{thm:ucbaire_hyperarithmetic_solution}.
\end{proof}

In particular, this shows that $\DS$ is not an arithmetic problem (i.e.\ $\DS\not\weireducible\mflim^{(n)}$, for any $n$). 

\begin{proposition}
    \thlabel{thm:CBaire_below_lim_cp_DS}
	$\CBaire \weiequiv \mflim\compproduct \DS$.
\end{proposition}
\begin{proof}
	The reduction $\mflim\compproduct \DS\weireducible \CBaire$ follows from the fact that both $\mflim$ and $\DS$ are reducible to $\CBaire$ and that $\CBaire$ is closed under compositional product.
	
	To prove the left-to-right reduction notice that, given a tree $T$, we can computably build the linear order $\KB(T)$. It is known that $[T]\neq\emptyset$ iff $\KB(T)$ is ill-founded (see e.g.\ \cite[Lem.\ V.1.3]{Simpson09}). Moreover, given a infinite descending sequence $\sequence{\sigma_n}{n\in\mathbb{N}}$ in $\KB(T)$, the sequence $\sequence{\sigma_n\concat 0^\omega}{n\in\mathbb{N}}$ converges to some $x\in [T]$, and therefore the claim follows.
\end{proof}

We can generalize the problem $\DS$ to the context of quasi-orders. It is easy to see that the problem of finding descending sequences in a quasi-order is Weihrauch equivalent to $\CBaire$. Indeed, on the one hand, being a descending sequence in a quasi-order $P$ is a $\Pi^{0,P}_1$ property. On the other hand, every tree, ordered by the prefix relation, is a partial order where the descending sequences provide arbitrarily long prefixes of a path.

When working with non-well quasi-orders, it is more natural to ask for bad sequences instead. 
\begin{definition}
\thlabel{def:bs}
	We define the multi-valued function $\BS\pmfunction{\mathbf{QO}}{\Baire}$ as
	\[ \BS(P):= \{ x\in\Baire \st (\forall i)(\forall j>i)(x(i) \not\preceq_P x(j)) \}, \] 
	where $\dom(\BS)$ is the set of quasi-orders that are not well-quasi-orders.
\end{definition}

It follows from the definition that every ill-founded linear order is a non-well quasi-order and that every bad sequence through an ill-founded linear order is indeed a descending sequence.

By expanding a bit on a classical argument we can prove that the two problems are uniformly equivalent.
\begin{proposition}
	\thlabel{thm:bs=ds}
	$\DS\weiequiv \BS$.
\end{proposition}
\begin{proof}
	See \ref{errata}.
\end{proof}

We will show that $\DS$ 
is quite weak in terms of uniform computational strength (a fortiori $\CBaire\not\weireducible\DS$). Let us first underline the following useful proposition.

\begin{proposition}
	\thlabel{thm:DS_cylinder}
	$\DS$ is a cylinder.
\end{proposition}
\begin{proof}
	Let $p\in\Baire$ and $L$ be an ill-founded linear order. Define 
	\begin{gather*}
		M:=\{ (p[n],n)\st n\in L \},\\
		(p[n],n) \le_M (p[m],m) \defiff n \le_L m.
	\end{gather*}
	It is easy to see that $M$ is computably isomorphic to $L$, and hence it is a valid input for $\DS$. In particular, letting $\sequence{(p[n_i],n_i)}{i\in\mathbb{N}}\in \DS(M)$, we have that $\sequence{n_i}{i\in\mathbb{N}}$ is a descending sequence in $L$ and $p = \bigcup_{i\in\mathbb{N}} p[n_i]$.
\end{proof}

\begin{definition}
	Let $\codedBound{\boldfaceGamma}\pmfunction{\boldfaceGamma(\mathbb{N}) }{\mathbb{N}}$ be the first-order problem that takes as input a finite $\boldfaceGamma$ subset of the natural numbers and returns a bound for it. Formally
	\begin{gather*}
		\dom(\codedBound{\boldfaceGamma}):=\{ A\in \boldfaceGamma(\mathbb{N}) \st (\forall^\infty n)(A(n)=0) \},\\ 
		\codedBound{\boldfaceGamma}(A):=\{ n\in\mathbb{N} \st (\forall m\ge n)(A(m)=0) \}.
	\end{gather*}
\end{definition}

The principle $\PiBound$ has been studied in \cite{KiharaADauriacChoice} under the name $\codedChoice{\boldfaceSigma^1_1}{cof}{\mathbb{N}}$: notice indeed that the reduction $\codedChoice{\boldfaceSigma^1_1}{cof}{\mathbb{N}}\strongweireducible\PiBound$ is trivial. On the other hand, given a finite $\boldfacePi^1_1$ subset $X$ of $\mathbb{N}$ we can consider the set 
\[ Y:=\{ n\in \mathbb{N} \st (\exists m\ge n)(m\in X) \}. \]
Clearly $Y$ is a $\boldfacePi^1_1$ initial segment of $\mathbb{N}$, and therefore $\setcomplement[\mathbb{N}]{Y}$ is a valid input for $\codedChoice{\boldfaceSigma^1_1}{cof}{\mathbb{N}}$. Moreover a name for $Y$ can be uniformly computed from a name of $X$ and $\codedChoice{\boldfaceSigma^1_1}{cof}{\mathbb{N}}(\setcomplement[\mathbb{N}]{Y})\subset \PiBound(X)$. This shows that $\PiBound\strongweireducible\codedChoice{\boldfaceSigma^1_1}{cof}{\mathbb{N}}$ and hence the two problems are (strongly) Weihrauch equivalent.

In other words, given an instance $X$ of $\PiBound$ we can, without loss of generality, assume that $X$ is an initial segment of $\mathbb{N}$.

\begin{proposition}
	\thlabel{thm:pibound<BS}
	$\PiBound\strictlyweireducible\DS$.
\end{proposition}
\begin{proof}
	Let $X$ be a $\boldfacePi^1_1$ initial segment of $\mathbb{N}$. By considering the Kleene-Brouwer ordering, we can think of a name for $X$ as a sequence $\sequence{L_n}{n\in\mathbb{N}}$ of linear orders s.t.\ $n\in X$ iff $L_n$ is well-founded.

	Define the linear order $L:=\bigcup_n \{n\}\times L_n$, ordered lexicographically. Notice that $L$ is ill-founded as $X$ is not all of $\mathbb{N}$. Moreover, for every $<_L$-descending sequence $\sequence{(n_i,a_i)}{i\in\mathbb{N}}$, we have that $n_0\in \PiBound(X)$. Indeed, for every $n\in X$ and every $a\in L_n$, the pair $(n,a)$ lies in the well-founded part of $L$.

	The fact that the reduction is strict follows from the fact that every solution to $\PiBound$ is computable, whereas there is a computable input for $\DS$ with no hyperarithmetic solution.
\end{proof}

We now show that $\firstOrderPart{\DS}\weiequiv \PiBound$. Let us first prove the following lemma, which will also be useful to prove \thref{thm:below_DS_below_lim}.

\begin{lemma}
	\thlabel{thm:Fs_construction}
	Suppose that $f$ is a problem which is Weihrauch reducible to $\DS$ via the computable maps $\Phi,\Psi$. For every $f$-instance $X$, let $\le^X$ be the linear order defined by $\Phi^X$. We can uniformly compute a sequence $\sequence{F_s}{s\in\mathbb{N}}$ of finite $<^X$-descending sequences s.t.\ (1) for every $s$, $\Psi^{X\oplus F_s}$ outputs some $j\in\mathbb{N}$; (2) for cofinitely many $s$, $F_s$ extends to an infinite $<^X$-descending sequence.
\end{lemma}
\begin{proof}
	Fix an $f$-instance $X$ and run $\Phi^X$ for $s$ steps. This produces a finite linear order $\le^X_s$. Define
	\begin{align*}
		D_s := \{F \subseteq\, \le^X_s \st {}&{} F \text{ is a} <^X_s\text{-descending sequence and } \length{F}\ge 1 \text{ and } \\  
			& \Psi^{X \oplus F} \text{ outputs some } j \in \mathbb{N} \text{ in }s \text{ steps}\}.
	\end{align*}
	Note that $D_s$ is finite and $t<s$ implies $D_t\subset D_s$. If $D_s\neq\emptyset$ we define $F_s$ to be the $<_\mathbb{N}$-least element of $D_s$ such that
	\[ \textstyle (\forall F \in D_s)\left(\min_{<^X}(F) \le^X_s \min_{<^X}(F_s)\right). \]
	This ensures that if any $F \in D_s$ extends to an infinite $<^X$-descending sequence, then so does $F_s$. Observe that $\sequence{F_s}{s}$ is uniformly computable from $X$. If $D_s=\emptyset$ we define $F_s:=F_t$ where $t$ is the first index greater than $s$ s.t.\ $D_t\neq\emptyset$. (We will show below that such $t$ exists, so we can computably search for it.)

	Notice that for cofinitely many $s$, $D_s\neq\emptyset$. Indeed, let $S$ be an infinite $<^X$-descending sequence (there must exist one because $<^X$ is a $\DS$-instance). Since $\Psi^{X \oplus S}$ outputs some $f$-solution $j$ of $X$, there is some finite nonempty initial segment $F$ of $S$ and some $t \in \mathbb{N}$ such that $\Psi^{X \oplus F}$ outputs $j$ in $t$ steps. Hence for all sufficiently large $s$, we have that $F \in D_s$. This shows that the sequence $\sequence{F_s}{s\in\mathbb{N}}$ is well-defined. Moreover, as already observed, for every $t\ge s$, $F_t$ extends to an infinite $<^X$-descending sequence.

	The fact that, for every $s$, $\Psi^{X\oplus F_s}$ outputs some $j\in\mathbb{N}$ follows from the definition of $D_s$.
\end{proof}

In particular, if $f$ has codomain $\mathbb{N}$ the above lemma implies that, for cofinitely many $s$, $\Psi^{X\oplus F_s}$ outputs some $f$-solution for $X$.

\begin{theorem} 
	\thlabel{thm:below_BS_below_Pi11Bound}
	$\firstOrderPart{\DS} \weiequiv \PiBound$.
\end{theorem}
\begin{proof}
	If $f \weireducible \PiBound$, then $f \weireducible \DS$ by \thref{thm:pibound<BS}. Since $\PiBound$ is first order, $f \weireducible \firstOrderPart{\DS}$.
	
	To prove the converse reduction, suppose that $f \weireducible \DS$ as witnessed by the maps $\Phi$ and $\Psi$. Given an $f$-instance $X$, let $\sequence{F_s}{s\in\mathbb{N}}$ be as in \thref{thm:Fs_construction}. Let $\le^X$ denote the linear order represented by $\Phi^X$. Define the following $\Pi^{1,X}_1$ set:
	\[ A := \{s \in \mathbb{N}\st F_s \notin \mathrm{Ext} \}, \]
	where $\mathrm{Ext}$ denotes the set of finite sequences that extend to an infinite $<^X$-descending sequence.

	Notice that $A$ is finite as, for cofinitely many $s$, $F_s$ is extendible. In particular $A$ is a valid instance of $\PiBound$ and, for every $b\in\PiBound(A)$, $F_b$ is extendible to an infinite $<^X$-descending sequence. By construction, $\Psi^{X\oplus F_b}$ commits to some $j\in\mathbb{N}$. The fact that $F_b$ is extendible guarantees that $j$ is a valid $f$-solution of $X$.
\end{proof}

\begin{corollary}
    \thlabel{thm:DS<CBaire}
	$\DS \strictlyweireducible \CBaire$.
\end{corollary}
\begin{proof}
	If $\CBaire\weireducible\DS$ then, by \thref{thm:fo_part_cbaire}, $\codedChoice{\boldfaceSigma^1_1}{}{\mathbb{N}}\weireducible \PiBound$. However, this would imply that $\parallelization{\codedChoice{\boldfaceSigma^1_1}{}{\mathbb{N}}}\weireducible \parallelization{\PiBound}$, contradicting \cite[Cor.\ 3.23]{KiharaADauriacChoice}.
\end{proof}

\begin{definition}
	Let $f\pmfunction{\mathbf{X}}{\mathbb{N}}$ be a multi-valued function. We say that $f$ is \textdef{upwards-closed} if whenever $n \in f(x)$, then $m \in f(x)$ for all $m > n$.
\end{definition}

It is straightforward from the definition that $\PiBound$ is upwards-closed.

\begin{lemma}
	\thlabel{prop:singlevalued_and_upclosed_then_CN}
	If $f$ is upwards-closed then $\DetPartX{f}{\mathbb{N}}\weireducible \CNatural$.
\end{lemma}
\begin{proof}
	Let $g$ be a single-valued function with codomain $\mathbb{N}$ and suppose that $g \weireducible f$ as witnessed by $\Phi, \Psi$. Given a name $p$ for a $g$-instance $x$, we use $\CNatural$ to guess some $n, t$ such that $\Psi(p,n)$ converges to some $k$ in at most $t$ steps, and such that for no $m > n$ it ever happens that $\Psi(p,m)$ converges to anything but $k$. 
	Since $f$ is upwards-closed and $g$ is single-valued, such $n,t$ must exist. Moreover, the associated $k$ is equal to $g(x)$.
\end{proof}

\begin{proposition}
    $\DetPartX{\PiBound}{\mathbb{N}}\weiequiv \DetPartX{\CNatural}{\mathbb{N}} \weiequiv \CNatural$, and therefore $\DetPartX{\DS}{\mathbb{N}}\weiequiv \CNatural$.
\end{proposition}
\begin{proof}
    Let us first notice that $\CNatural\weiequiv\UCNatural$ (\cite[Prop.\ 6.2]{BdBPLow12}) and therefore $\DetPartX{\CNatural}{\mathbb{N}} \weiequiv \CNatural$. The fact that $\DetPartX{\PiBound}{\mathbb{N}}\weireducible \CNatural$ follows from \thref{prop:singlevalued_and_upclosed_then_CN}. To prove the converse reduction it is enough to show that $\UCNatural\weireducible \PiBound$. 
    
    Let $\sequence{n_i}{i\in\mathbb{N}}$ be an enumeration of the complement of $\{x\}\subset \mathbb{N}$. Define 
    \begin{gather*}
        m(s):=\min\{j\in\mathbb{N}  \st (\forall i<s)(n_i \neq j) \},\\
         A:=\{ s \in\mathbb{N} \st (\exists t>s )(m(t)\neq m(s)) \}.
    \end{gather*}
    Clearly $\lim_{s\to\infty} m(s) =x$, which implies that $A$ is finite. Since $m$ is computable (relative to $\sequence{n_i}{i\in\mathbb{N}}$), $A$ is a valid input for $\PiBound$. Moreover, for every $b\in\PiBound(A)$ we have $m(b)=x$.
    
	This implies that $\CNatural\weireducible \DetPartX{\DS}{\mathbb{N}}$. To conclude the proof we notice that, for every single-valued $g$ with codomain $\mathbb{N}$ we have
	\[  g\weireducible \DS \Rightarrow g \weireducible \PiBound \Rightarrow g\weireducible \DetPartX{\PiBound}{\mathbb{N}}\weiequiv \CNatural.  \qedhere\]
\end{proof}

Notice that $\PiBound\not\weireducible \CNatural$: indeed $\parallelization{\CNatural}\weiequiv \mflim$, while $\UCBaire\strictlyweireducible\parallelization{\PiBound}$ (see \thref{thm:delta11-DS<sigma11-DS}). This implies that $\DetPartX{\PiBound}{\mathbb{N}}\strictlyweireducible \PiBound$. In this regard, we observe the following:

\begin{proposition}
	The Weihrauch degree of $\CNatural$ is the highest Weihrauch degree containing both of the following:
	\begin{enumerate}
		\item a representative which is single-valued and has codomain $\mathbb{N}$;
		\item a representative which is upwards-closed.
	\end{enumerate}
\end{proposition}
\begin{proof}
	To prove that $\CNatural$ satisfies point 1, consider $\UCNatural$, which is Weihrauch equivalent to $\CNatural$ (\cite[Prop.\ 6.2]{BdBPLow12}). To prove that $\CNatural$ satisfies point $2$, consider the problem $\codedBound{\boldfaceSigma^0_1}$ that produces a bound for a finite $\boldfaceSigma^0_1$ subset of $\mathbb{N}$. Clearly $\codedBound{\boldfaceSigma^0_1}$ is upwards closed. The reduction $\codedBound{\boldfaceSigma^0_1}\weireducible\CNatural$ follows from the fact that, for every $A\in \dom(\codedBound{\boldfaceSigma^0_1})$, the set 
	\[ \{ n\in\mathbb{N} \st (\forall m \ge n)( m\notin A ) \} \]
	is a $\Pi^{0,A}_1$ subset of $\codedBound{\boldfaceSigma^0_1}(A)$. To prove the converse reduction, let $p$ be a name for some $B\in \dom(\CNatural)$. Define $m(s)$ to be the least number not enumerated in $p$ by stage $s$. Clearly $\lim_{s\to\infty} m(s) = \min B$. In particular this implies that there are only finitely many stages $s$ s.t.\ $m(s)\neq \min B$. Using $\codedBound{\boldfaceSigma^0_1}$ we can obtain a stage $b$ s.t.\ $m(b)=\min B$, hence solving $\CNatural$.

	Finally the maximality of $\CNatural$ follows from \thref{prop:singlevalued_and_upclosed_then_CN}: indeed suppose $f: \textbf{X} \to \mathbb{N}$ is Weihrauch equivalent to some $g$ which is upwards-closed. By \thref{prop:singlevalued_and_upclosed_then_CN}, we have $\DetPartX{g}{\mathbb{N}}\weireducible \CNatural$. By definition of $\DetPart{\cdot}$, we have $f \weireducible \DetPartX{g}{\mathbb{N}}$, hence $f \weireducible \CNatural$.
\end{proof}

Let us now characterize the deterministic part of $\DS$.

\begin{theorem}
	\thlabel{thm:below_DS_below_lim}
	$\DetPart{\DS}\weiequiv\mflim$.
\end{theorem}
\begin{proof}
	Let us first prove that $\mflim\weireducible\DS$. Let $\mfJ$ be the Turing jump operator, i.e.\ $\mfJ(p)(e)=1$ iff $\varphi_e^p(e)$ halts. It is known that $\mfJ\strongweiequiv \mflim$ (see \cite[Thm.\ 6.7]{BGP17}). By relativizing the construction in \cite[Lem.\ 4.2]{MarconeShore2011} we have that, for every $p$, we can $p$-computably build a linear order $L$ of type $\omega+\omega^*$ s.t.\ every descending sequence through $L$ computes $\mfJ(p)$. This shows that $\mflim\weiequiv\mfJ\weireducible \DS$.
	
	To prove that $\DetPart{\DS} \weireducible \mflim$, suppose that $f \pfunction\textbf{X} \to \Baire$ is single-valued and $f\weireducible\DS$ as witnessed by the maps $\Phi$, $\Psi$. For every $n$, define $f_n$ by $f_n(X):=f(X)(n)$. The maps $\Phi$ and $\Psi$ witness that $f_n\weireducible\DS$ as well (modulo a trivial coding). Given an $f$-instance $X$, consider the sequences $\sequence{F_{s,n}}{s\in\mathbb{N}}$ obtained by applying \thref{thm:Fs_construction} to each $f_n$. Define the sequence $\sequence{p_s}{s\in\mathbb{N}}$ in $\Baire$ as $p_s(n):= \Psi^{X\oplus F_{s,n}}(0)$. Notice that, by \thref{thm:Fs_construction}, for every $n$, $\Psi^{X\oplus F_{s,n}}$ outputs some number, therefore $p_s(n)$ is well-defined and is uniformly computable from $X$. Moreover, since $f_n$ is single-valued and, for cofinitely many $s$, $F_{s,n}$ is extendible, the sequence $\sequence{\Psi^{X\oplus F_{s,n}}(0)}{s\in\mathbb{N}}$ is eventually constant and equal to $f_n(X)$. In particular this shows that, letting $p:=\lim_{s\to\infty} p_s$, for each $n$ we have $p(n)=f_n(X)$, i.e.\ $p=f(X)$.
\end{proof}

This result shows that, despite the fact that $\DS$ can have very complicated solutions, it is rather weak from the uniform point of view. In fact, its lower Weihrauch cone misses many arithmetic problems. In particular we have:

\begin{corollary}
	\thlabel{thm:lpo'_incomp_BS}
	$\DS\weiincomparable\LPO'$.
\end{corollary}
\begin{proof}
    Since $\LPO$ is single-valued, so is $\LPO'$. Since $\LPO'\not\weireducible \mflim$ (see \cite[Cor.\ 12.3 and Thm.\ 12.7]{BolWei11}), it follows from \thref{thm:below_DS_below_lim} that $\LPO'\not\weireducible \DS$. On the other hand, $\DS\not\weireducible\LPO'$, as $\LPO'$ always has computable solutions.
\end{proof}

Notice that \thref{thm:below_DS_below_lim} implies also that $\CBaire \not\weireducible \CCantor \compproduct \DS$. Indeed, on the one hand  $\DetPart{\CBaire}\weiequiv \UCBaire$ (\thref{thm:detpart_cbaire}), while, on the other hand, by \thref{thm:det(f*g)<=det(f)*g}
if $f$ is single-valued and $f\weireducible \CCantor\compproduct \DS$ then $f\weireducible \DS$ (as $\DetPart{\CCantor}\weiequiv\id{}$) and hence $\DetPart{\CCantor\compproduct\DS}\weiequiv\DetPart{\DS}\weiequiv\mflim$.

Using \thref{thm:lpo'_incomp_BS} we can prove that $\DS$ is not closed under (parallel) product:

\begin{theorem}
    \thlabel{thm:DS_not_closed_under_product}
	$\LPO'\weireducible \DS\times \mflim$ and therefore $\DS$ is not closed under product.
	\thlabel{thm:lpo'_DS_times_lim}
\end{theorem}
\begin{proof}
	Let $\sequence{p_n}{n\in\mathbb{N}}$ be a sequence in $\Baire$ converging to an instance $p$ of $\LPO$. For each $s$ define 
	\[ g(s)= \begin{cases}
		i+1 & \text{if } i\le s \land p_s(i)\neq 0 \land (\forall j<i)(p_s(j)=0),\\
		0 & \text{otherwise.}
	\end{cases}\]
	Let us define a linear order $L$ inductively: at stage $s=0$ we put $0$ into $L$. At stage $s+1$ we do the following:
	\begin{enumerate}
		\item if $g(s)=g(s+1)$ we put $2(s+1)$ immediately below $2s$;
		\item if $g(s)\neq g(s+1)$ and $g(s+1)=0$ we put $2(s+1)$ at the bottom;
		\item if $g(s)\neq g(s+1)$ and $g(s+1)>0$  we put $2(s+1)$ at the top and we put $2s+1$ immediately above $0$.
	\end{enumerate}

	This construction produces a linear order on a computable subset of $\mathbb{N}$. It is clear that $g$ and $L$ are uniformly computable in $\sequence{p_n}{n\in\mathbb{N}}$. Notice that if $\LPO(p)=1$ then there is an $s$ s.t.\ for every $t\ge s$, $g(t)=g(s)$ (this follows by definition of limit in the Baire space). In particular, $L$ has order type $n+\omega^*$. On the other hand, if $\LPO(p)=0$ we distinguish three cases: if $g(s)$ is eventually constantly $0$ then $L$ has order type $\omega^*$. If there are infinitely many $s$ s.t.\ $g(s)>0$ then $g$ is unbounded (because for each $i$, $\lim_s p_s(i) = p(i) = 0$ so $g$ eventually stays above $i$). In particular, if there are infinitely many $s$ and infinitely many $t$ s.t.\ $g(s)=0$ and $g(t)>0$ then $L$ has order type $\omega^*+\zeta$, where $\zeta:=\omega^*+\omega$ is the order type of the integers. If instead $g(s) > 0$ for all sufficiently large $s$, then $L$ has order type $n+\zeta$. In all cases, $L$ is ill-founded.
	
	We consider the input $(L,\sequence{p_n}{n\in\mathbb{N}})$ for $\DS\times\mflim$.	Given an $<_L$-descending sequence $\sequence{q_n}{n\in\mathbb{N}}$, we compute a solution for $\LPO'(\sequence{p_n}{n\in\mathbb{N}})=\LPO(p)$ as follows: if $q_0$ is odd or $g(q_0/2)=0$ then we return $0$, otherwise we return $p(i)$ where $i$ is s.t.\ $g(q_0/2)=i+1$.
	
	Notice that if $\LPO(p)=1$ then the $\omega^*$ part of $\leq_L$ is the final segment of the even numbers that starts with the first index $2s$ s.t.\ for every $t\ge s$, $g(t)=i+1$ and $p(i)=1$. In particular every $<_L$-descending sequence starts with some even $q_0$ s.t.\ $g(q_0/2)>0$.
	On the other hand, if $\LPO(p)=0$ then, by definition of $\LPO$, we have that $p=0^\mathbb{N}$. In this case, the above procedure must return $0$ so it produces the correct solution. This proves that $\LPO' \weireducible \DS \times \mflim$.

	The fact that $\DS$ is not closed under product follows from the fact that $\mflim \weireducible \DS$ (\thref{thm:below_DS_below_lim}) and \thref{thm:lpo'_incomp_BS}.
\end{proof}

\subsection{Combinatorial principles on linear orders}

We introduce the following notation to phrase many combinatorial principles from reverse mathematics as multi-valued functions.

\begin{definition}
	Let $\findC{X}{Y}\pmfunction{\mathbf{LO}}{\mathbf{LO}}$ be the partial multi-valued function defined as 
	\[ \findC{X}{Y}(L):=\{ M\in \LO\st M\subset L \text{ and } \operatorname{ordtype}(M)\in Y \}, \]
	with domain being the set of $L\in \LO$ s.t.\  $\operatorname{ordtype}(L)\in X$ and there is some $M\subset L$ s.t.\ $\operatorname{ordtype}(M)\in Y$.

	Similarly we define $\findS{X}\pmfunction{\mathbf{LO}}{\Baire}$ to be the partial multi-valued function that takes as input a countable linear order $L$ s.t.\ $\operatorname{ordtype}(L)\in X$ and produces a string $\coding{b,x_0,x_1,\hdots}$ s.t.\ $b\in\{0,1\}$ and, for all $i$, if $b=0$ then $x_i <_L x_{i+1}$ while if $b=1$ then $x_{i+1}<_L x_i$.

	If $X$ or $Y$ is not specified, we assume that it contains every countable order type.
\end{definition}
There is an extensive literature that studies the ``ascending/descending sequence principle'' ($\mathrm{ADS}$) and the ``chain/antichain principle'' ($\mathrm{CAC}$) (see e.g.\ \cite{HS07, JKLLS09}). These principles and, several of their variations, have been studied from the point of view of Weihrauch reducibility in \cite{ADSS17}. 

Notice that, in particular, the problem $\ADS$ (given a linear order, produce an infinite ascending sequence or infinite descending sequence) corresponds to $\findS{}$. Similarly  the problem $\GSADS$ (given a stable --- i.e.\ of order type $\omega+n,n+\omega^*$ or $\omega+\omega^*$ --- linear order, produce an infinite ascending or descending sequence), corresponds to $\findS{X}$, where $X =\{\omega+n,n+\omega^*,\omega+\omega^*\}$. 
\begin{proposition}
	\thlabel{thm:lpo'<=findS}
	$\LPO'\weireducible \findS{\{\omega,n+\omega^*\}}$.
\end{proposition}
\begin{proof}
	Let $\sequence{p_i}{i\in\mathbb{N}}$ be a sequence in $\Baire$ converging to an instance $p$ of $\LPO$. For every $s\in\mathbb{N}$ we define (as we did in the proof of \thref{thm:lpo'_DS_times_lim})
	\[ g(s)= \begin{cases}
		i+1 & \text{if } i\le s \land p_s(i)\neq 0 \land (\forall j<i)(p_s(j)=0),\\
		0 & \text{otherwise.}
	\end{cases} \]
	Let us define a linear order $\le_L$ on $\mathbb{N}$ inductively: for each stage $s$ we define a linear order on $\{0,\hdots,s\}$. At stage $s=0$ there are no decisions to make. At stage $s+1$ we do the following:
	\begin{enumerate}
		\item if $0=g(s)=g(s+1)$ we put $s+1$ immediately above $s$;
		\item if $0<g(s)=g(s+1)$ we put $s+1$ immediately below $s$;
		\item if $g(s)\neq g(s+1)$ we put $s+1$ at the top.
	\end{enumerate}
	It is clear that $g$ and $\leq_L$ are uniformly computable in $\sequence{p_n}{n\in\mathbb{N}}$. Notice that if $\LPO(p)=1$ then there is an $s$ s.t.\ for every $t\ge s$, $g(t)=i+1$, where $i$ is the smallest integer s.t.\ $p(i)=1$ (this follows by definition of limit in the Baire space). In particular, $\le_L$ has order type $n+\omega^*$. On the other hand, if $\LPO(p)=0$ then $g$ is either eventually constantly $0$ or unbounded. In both cases the linear order $\le_L$ has order type $\omega$. In other words $(\mathbb{N},\le_L)$ has order type $\omega$ iff $\LPO'(\sequence{p_i}{i\in\mathbb{N}})=0$. Since the output of $\findS{\{\omega,n+\omega^*\}}((\mathbb{N},\le_L))$ comes with an indication of the order type of the solution, this defines a reduction from $\LPO'$ to $\findS{\{\omega,n+\omega^\ast\}}$.
\end{proof}

\begin{corollary}
    \thlabel{thm:gsads_incomp_ds}
	$\findS{\{\omega, n+\omega^*\}} \weiincomparable \DS$, and hence $\GSADS \weiincomparable \DS$.
\end{corollary}
\begin{proof}
	The fact that $\findS{\{\omega, n+\omega^*\}} \not\weireducible \DS$ follows from \thref{thm:lpo'<=findS} and the fact that $\LPO'\not\weireducible\DS$ (\thref{thm:lpo'_incomp_BS}). Moreover, since $\findS{\{\omega, n+\omega^*\}}$ is a restriction of $\GSADS$, we have $\GSADS\not\weireducible\DS$.

	To show that the converse reduction cannot hold it is enough to notice that $\GSADS$ is an arithmetic problem, while $\DS\not\weireducible\UCBaire$ (\thref{thm:bs_not_<=_ucbaire}).
\end{proof}
In particular this implies that $\ADS$, as well as the stable chain/antichain principle $\mathsf{SCAC}$, and the weakly stable chain/antichain principle $\mathsf{WSCAC}$ are Weihrauch incomparable with $\DS$ (as they are all arithmetic problems, and $\GSADS$ is reducible to all of them, see \cite{ADSS17}).

\begin{proposition}
	$\findC{}{\{ \omega,n+\omega^*\}} \weireducible \DS$.
\end{proposition}
\begin{proof}
	Given a linear order $(L,\le_L)$ we can computably build the linear order $Q:=L+L^*$. Formally we define $(Q,\le_Q)$ as $Q:=\{0\}\times L \cup \{1\}\times L$ and 
	\[ (a,p)\le_Q (b,q) \defiff a<b \lor (a=b=0 \land p\le_L q) \lor (a=b=1 \land q \le_L p). \]
	Notice that $Q$ is always ill-founded, hence it is a valid input for $\DS$. Given $\sequence{q_i}{i\in\mathbb{N}}\in \DS(Q)$, we computably build the sequence $\sequence{x_i}{i\in\mathbb{N}}$ defined by $x_i:=\pi_1 q_i$ where $\pi_i:= (a_0,a_1)\mapsto a_i$. 
	
	We distinguish $3$ cases:
	\begin{enumerate}
		\item if $\pi_0 q_i =0$ for every $i$ then $\sequence{x_i}{i\in\mathbb{N}}$ is an $\omega^*$-sequence in $L$;
		\item if $\pi_0 q_i =1$ for every $i$ then $\sequence{x_i}{i\in\mathbb{N}}$ is an $\omega$-sequence in $L$;
		\item if there is a $k$ s.t.\ for all $i<k$ we have $\pi_0 q_i = 1$ and for all $j\ge k$ we have $\pi_0 q_j = 0$ then, by point $1$, $\sequence{x_j}{j\ge k}$ is an $\omega^*$-sequence in $L$, hence $\sequence{x_i}{i\in\mathbb{N}}$ is of type $n+\omega^*$, with $n\le k$.
	\end{enumerate}
	In any case the sequence $\sequence{x_i}{i\in\mathbb{N}}$ is a valid solution for $\findC{}{\{ \omega,n+\omega^*\}}$.
\end{proof}

\subsection{Relations with Ramsey theorems}
We now explore the relations between $\DS$ and Ramsey's theorem for $n$-tuples and $k$ colors. Let us recall the basic definitions.

\begin{definition}
	For every $A\subset \mathbb{N}$, let $[A]^n:=\{ B\subset A\st |B|=n\}$ be the set of subsets of $A$ with cardinality $n$. A map $c\colon[\mathbb{N}]^n\to k$ is called a \textdef{$k$-coloring} of $[\mathbb{N}]^n$, where $k\ge 2$. An infinite set $H$ s.t.\ $c([H]^n)=\{i\}$ for some $i<k$ is called a \textdef{homogeneous solution} for $c$, or simply \textdef{homogeneous}.
	
	The set $\mathcal{C}_{n,k}$ of $k$-colorings of $[\mathbb{N}]^n$ can be seen as a represented space, where a name for a coloring $c$ is the string $p\in\Baire$ s.t.\ for each $(i_0,\hdots,i_{n-1})\in[\mathbb{N}]^n$, $p(\coding{i_0,\hdots,i_{n-1}})=c(i_0,\hdots,i_{n-1})$.
	
	We define $\RT{n}{k}\mfunction{\mathcal{C}_{n,k}}{\Cantor}$ as the total multi-valued function that maps a coloring $c$ to the set of all homogeneous sets for $c$. Similarly we define $\RT{n}{\mathbb{N}}\mfunction{\bigcup_{k\ge 1}\mathcal{C}_{n,k} }{\Cantor}$ as $\RT{n}{\mathbb{N}}(c):=\RT{n}{k}(c)$, where $k-1$ is the maximum of the range of $c$. Note that the input for $\RT{n}{\mathbb{N}}$ does not include information on which color appears in the range of the coloring.

	We also define $\cRT{n}{k}\mfunction{\mathcal{C}_{n,k}}{k}$ as the multi-valued function that produces only the color of a homogeneous solution. We define $\cRT{n}{\mathbb{N}}$ analogously. 
\end{definition}

Notice that $\cRT{n}{k}\weiequiv \RT{n}{k}$ iff $n=1$. Indeed the output of $\cRT{n}{k}$ is always computable, while for $n>1$ there are computable $k$-colorings with no computable homogeneous solutions. Similarly $\cRT{n}{\mathbb{N}}\weiequiv \RT{n}{\mathbb{N}}$ iff $n=1$. Moreover the equivalence cannot be lifted to a strong Weihrauch equivalence. Indeed $\RT{1}{k}$ and $\cRT{1}{k}$ are incomparable from the point of view of strong Weihrauch reducibility. The uniform computational content of Ramsey's theorems is well-studied (see e.g.~\cite{BRramsey17,DDHMS16, DGHPP18,patey}).

In comparing $\RT{n}{k}$ with $\DS$, we immediately notice that $\RT{2}{2}\not\weireducible \DS$. This follows from the fact that $\ADS\weireducible \RT{2}{2}$ (see e.g.\ \cite{HS07}), while $\ADS\not\weireducible\DS$ (see the remarks after \thref{thm:gsads_incomp_ds}). Hence $\RT{n}{k} \not\weireducible \DS$ for all $n,k \geq 2$.

\begin{proposition}
	\thlabel{thm:RT1N<Pi11Bound}
	$\RT{1}{\mathbb{N}}\strictlyweireducible\PiBound$, and hence $\RT{1}{\mathbb{N}}\strictlyweireducible\DS$.
\end{proposition}
\begin{proof}
	Given a coloring $c\colon\mathbb{N}\to k$, consider the $\Sigma^{0,c}_2$ set
	\[ X:= \{ n \in \mathbb{N} \st (\forall^\infty j)(c(n) \neq c(j))  \}. \]
	It is easy to see that $X$ is finite, as $\ran(c)\subset k$ and if there is no $c$-homogeneous set with color $i$ then there are finitely many $j\in\mathbb{N}$ s.t.\ $c(j)=i$. In particular, given a bound $b$ for $X$ there is a homogeneous solution with color $c(b)$.
	
	The separation follows from the fact that $\PiBound\not\weireducible\UCBaire$ (as $\parallelization{\PiBound}\not\weireducible\UCBaire$, see \cite[Fact 3.25]{KiharaADauriacChoice}), while $\RT{1}{\mathbb{N}}\strictlyweireducible\UCBaire$ (in particular $\RT{1}{\mathbb{N}}\strictlyweireducible \CNatural'$, see \cite[Prop.\ 7.2 and Cor.\ 7.6]{BRramsey17}). The fact that $\RT{1}{\mathbb{N}}\strictlyweireducible\DS$ follows from $\PiBound\strictlyweireducible\DS$ (\thref{thm:pibound<BS}).
\end{proof}

We now show that $\RT{1}{\mathbb{N}}$ is the strongest problem among those that are reducible to $\DS$ and whose instances always have finitely many solutions.

\begin{definition}
	Let $f\pmfunction{\mathbf{X}}{\mathbb{N}}$. We say that $f$ is \textdef{pointwise finite} if, for each $x\in\dom(f)$, $|f(x)|$ is finite.
\end{definition}
Notice that $\cRT{1}{k}$ and $\cRT{1}{\mathbb{N}}$ are pointwise finite, as for each $k$-coloring $c$ we have $|\cRT{1}{k}(c)|=|\cRT{1}{\mathbb{N}}(c)|\le k$.

\begin{lemma}
	\thlabel{thm:f<upwards_closed_f<RT1N}
	Let $g$ be upwards-closed and let $f$ be pointwise finite. If $f\weireducible g$ then $f\weireducible \RT{1}{\mathbb{N}}$.
\end{lemma}
\begin{proof}
    Suppose that $f \weireducible g$ as witnessed by $\Phi, \Psi$. Let $p$ be the name for the $f$-instance $x$ we are given.
	
	We define a coloring $c$ as follows: we dove-tail all computations $\Psi(p,n)$ for $n \in \mathbb{N}$. Whenever some computation converges to some $j \in \mathbb{N}$, we define $c(i):=j$ where $i$ is the first element on which $c$ is not defined yet. Since $g$ is upwards-closed, we know that for all but finitely many $n$, $\Psi(p,n)$ has to converge to some $j_n \in f(x)$. This implies that $\ran(c)$ contains only finitely many distinct elements. Moreover, any element repeating infinitely often is a correct solution to $f(x)$, therefore we can find a $y\in f(x)$ by applying $\RT{1}{\mathbb{N}}$ to $c$ and returning the color of the solution.
\end{proof}

\begin{theorem}
	\thlabel{thm:f<DS_iff_f<RT1N}
	If $f$ is pointwise finite then $f\weireducible \DS$ iff $f\weireducible \RT{1}{\mathbb{N}}$.
\end{theorem}
\begin{proof}
	The right-to-left implication always holds as $\RT{1}{\mathbb{N}}\strictlyweireducible\DS$ (\thref{thm:RT1N<Pi11Bound}). On the other hand, if $f$ is pointwise finite and $f\weireducible\DS$ then, by \thref{thm:below_BS_below_Pi11Bound} we have $f\weireducible\PiBound$. Since $\PiBound$ is upwards-closed, by \thref{thm:f<upwards_closed_f<RT1N} we have $f\weireducible\RT{1}{\mathbb{N}}$.
\end{proof}

By \thref{thm:f<upwards_closed_f<RT1N} we also have the following:

\begin{proposition}
	The Weihrauch degree of $\RT{1}{\mathbb{N}}$ is the highest Weihrauch degree such that:
\begin{enumerate}
	\item it contains a representative which is pointwise finite;
	\item it is Weihrauch reducible to some problem which is upwards-closed.
\end{enumerate}
\end{proposition}
\begin{proof}
	Point 1 holds because $\cRT{1}{\mathbb{N}}$ is pointwise finite and $\cRT{1}{\mathbb{N}}\weiequiv \RT{1}{\mathbb{N}}$. Point 2 holds because $\RT{1}{\mathbb{N}}\strictlyweireducible\PiBound$ (\thref{thm:RT1N<Pi11Bound}) and $\PiBound$ is upwards-closed. Finally, the maximality follows from \thref{thm:f<upwards_closed_f<RT1N}.
\end{proof}

\begin{lemma}
	If $f$ is upwards-closed and $f\weireducible \RT{1}{\mathbb{N}}$ then $f\weireducible\CNatural$.
\end{lemma}
\begin{proof}
	Recall that $\RT{1}{\mathbb{N}}\weiequiv\cRT{1}{\mathbb{N}}$ and let $\Phi,\Psi$ be two computable maps witnessing $f\weireducible\cRT{1}{\mathbb{N}}$. Let $p$ be a name for some $x\in\dom(f)$ and let $c$ be the coloring represented by $\Phi(p)$. We define the following $\Pi^{0,p}_1$ set 
	\begin{align*}
		A:= \{\coding{n,c_0,\hdots,c_k,s}\st {}&{} (\forall i)(\exists j\le k)(c(i)=c_j) \text{ and }\\
			& (\forall j\le k)(\exists i<s)(c(i)=c_j) \text{ and }\\
			& (\forall j\le k)(\Psi(p,c_j)\defined \rightarrow \Psi(p,c_j)\le n) \}.
	\end{align*} 
	Notice that, if $\coding{n,c_0,\hdots,c_k,s}\in A$ then, by the first two conditions, there is a $j\le k$ s.t.\ $c_j$ is a valid solution for $\cRT{1}{\mathbb{N}}(c)$. In particular $\Psi(p,c_j)\downarrow$ and is a correct solution for $f(x)$ (as $\Phi$ and $\Psi$ witness that $f \weireducible \cRT{1}{\mathbb{N}}$). Since $f$ is upwards-closed, every number greater than $\Psi(p,c_j)$ is a valid solution. In particular, the third condition implies that $n\ge \Psi(p,c_j)$ and therefore $n\in f(x)$.
\end{proof}
Notice that the previous lemma provides an alternative proof for the fact that $\PiBound\not\weireducible \RT{1}{\mathbb{N}}$, as $\PiBound\not\weireducible\CNatural$.

If we consider only bounded pointwise finite functions, we can improve \thref{thm:f<DS_iff_f<RT1N} by replacing $\RT{1}{\mathbb{N}}$ with $\RT{1}{k}$.
\begin{lemma}
	\thlabel{thm:f<RT1N_iff_f<RT1k}
	If $f$ has codomain $k$, then $f \weireducible \RT{1}{\mathbb{N}}$ iff $f \weireducible \RT{1}{k}$.
\end{lemma}
\begin{proof}
	The right-to-left implication is trivial, so let us prove the left-to-right one. Since $\RT{1}{\mathbb{N}} \weireducible \PiBound$ and $\PiBound$ is upwards-closed, it suffices to show that if $g$ is upwards-closed and $f \weireducible g$, then $f \weireducible \RT{1}{k}$. The proof closely follows the one of \thref{thm:f<upwards_closed_f<RT1N}. Suppose that $f \weireducible g$ as witnessed by $\Phi, \Psi$. Let $p$ be the name for the $f$-instance $x$ we are given. We define a $k$-coloring $c$ as follows: we dove-tail all computations $\Psi(p,n)$ for $n \in \mathbb{N}$. Whenever some computation converges to some $j < k$, we define $c(i):=j$ where $i$ is the first element on which $c$ is not defined yet. Since $g$ is upwards-closed, we know that for all but finitely many $n$, $\Psi(p,n)$ has to converge to some $j_n < k$ which lies in $f(x)$. Moreover, any element repeating infinitely often is a correct solution to $f(x)$, therefore we can find a $y\in f(x)$ by applying $\RT{1}{k}$ to $c$ and returning the color of the solution.
\end{proof}

\begin{theorem}
	\thlabel{thm:below_DS_below_RT1k}
	If $f$ has codomain $k$, then $f \weireducible \DS$ iff $f \weireducible \RT{1}{k}$.
\end{theorem}
\begin{proof}
	The right-to-left implication always holds as $\RT{1}{k}\weireducible\RT{1}{\mathbb{N}}$ trivially and $\RT{1}{\mathbb{N}}\strictlyweireducible\DS$ (\thref{thm:RT1N<Pi11Bound}). The left-to-right implication follows from \thref{thm:f<DS_iff_f<RT1N} and \thref{thm:f<RT1N_iff_f<RT1k}.
\end{proof}

To conclude the section we notice how we can improve the results if we restrict our attention to single-valued functions. Define the problem $\mflim_k \pfunction k^{\mathbb{N}} \to k$ as the limit in the discrete space $k$. 

\begin{lemma}
	\thlabel{thm:f<limk_iff_f<RT1k}
	If $f$ has codomain $k$ and is single-valued, then $f \weireducible \mflim_k$ iff $f \weireducible \RT{1}{k}$.
\end{lemma}
\begin{proof}
	The left-to-right implication is trivial as $\mflim_k\weireducible\RT{1}{k}$. To prove the converse direction recall that $\RT{1}{k}\weiequiv\cRT{1}{k}$ and let the reduction $f\weireducible\cRT{1}{k}$ be witnessed by the maps $\Phi,\Psi$. Let $p$ be a name for some $x\in\dom(f)$ and let $c$ be the coloring represented by $\Phi(p)$. Notice that, since $f$ is single-valued, for every solution $j\in \cRT{1}{k}(c)$ we have $\Psi(p,j)=f(x)$. Furthermore, since the range of $c$ is finite, there are only finitely many $i$ such that $c(i)$ is not a solution. If we then define 
	\[ n_i:= \begin{cases}
		\Psi(p,c(i)) & \text{if } \Psi(p,c(i)) \text{ converges in }i\text{ steps and } \Psi(p,c(i))<k, \\
		0	&	\text{otherwise,}
	\end{cases}\]
	we have that the sequence $\sequence{n_i}{i\in\mathbb{N}} \in k^{\mathbb{N}}$ converges to $f(x)$. Therefore we can use $\mflim_k$ to obtain $f(x)$.
\end{proof}

\begin{theorem}
	\thlabel{thm:below_DS_below_lim_k}
	If $f$ has codomain $k$ and is single-valued, then $f \weireducible \mflim_k$ iff $f \weireducible \DS$.
\end{theorem}
\begin{proof}
	The left-to-right implication follows from the fact that $\mflim\strictlyweireducible\DS$ (\thref{thm:below_DS_below_lim}), while the other direction follows from \thref{thm:below_DS_below_RT1k} and \thref{thm:f<limk_iff_f<RT1k}.
\end{proof}

\section{Presentation of orders}
\label{sec:presentations}

In this section we study how the presentation of a linear/quasi order can influence the uniform computational strength of the problems $\DS$ and $\BS$.

\begin{definition}
	For every $\boldfaceGamma \in \{ \boldfaceSigma^0_k, \boldfacePi^0_k, \boldfaceDelta^0_k, \boldfaceSigma^1_1, \boldfacePi^1_1, \boldfaceDelta^1_1 \}$ we define the problem $\codedDS{\boldfaceGamma}\pmfunction{\boldfaceGamma(\mathbf{LO})}{\Baire}$ as $\codedDS{\boldfaceGamma}(L) := \DS(L)$. Similarly we define $\codedBS{\boldfaceGamma}\pmfunction{\boldfaceGamma(\mathbf{QO})}{\Baire}$ as $\codedBS{\boldfaceGamma}(P) := \BS(P)$.
\end{definition}

It is not the case that $\codedDS{\boldfaceGamma} \weiequiv \codedBS{\boldfaceGamma}$ in general. In particular, we will show that $\codedBS{\boldfaceSigma^0_k} \not\weireducible \codedDS{\boldfaceSigma^0_k}$ (\thref{thm:Sigma0kDS_LQO=Delta0_k+1DS_LO}) and $\codedBS{\boldfaceSigma^1_1} \not\weireducible \codedDS{\boldfaceSigma^1_1}$ (\thref{thm:Pi11DS<Pi11BS}).

Furthermore, we strengthen \thref{thm:DS<CBaire} by showing that $\codedDS{\boldfaceSigma^1_1} \strictlyweireducible \CBaire$ (\thref{thm:atr2_not<=_Sigma11DS_LO}). In other words, even if we are allowed to feed $\DS$ a code for a $\boldfaceSigma^1_1$ linear ordering, we still cannot compute $\CBaire$. On the other hand, we already showed that if we are allowed to perform a relatively small amount of post-processing (namely $\mflim$) on the output of $\DS$, then we can compute $\CBaire$ (\thref{thm:CBaire_below_lim_cp_DS}). In particular, the use of $\mflim$ absorbs any difference in uniform strength between $\DS$ and $\codedDS{\boldfaceSigma^1_1}$ and collapses the whole hierarchy (up to $\codedDS{\boldfaceSigma^1_1}$) to $\CBaire$.

Many of our separations are derived by analyzing the first-order part of the problems in question, or more generally by characterizing the problems satisfying certain properties (such as single-valuedness or having restricted codomain) which lie below the problems in question. On the contrary, we prove \thref{thm:atr2_not<=_Sigma11DS_LO} using very different techniques due to Angl\`es d'Auriac and Kihara \cite{KiharaADauriacChoice}.

Before beginning our analysis, we record some preliminary observations. Note that $\DS=\codedDS{\boldfaceDelta^0_1}$ and $\BS=\codedBS{\boldfaceDelta^0_1}$. It is straightforward to see that, for every $\boldfaceGamma$, $\codedDS{\boldfaceGamma} \weireducible \codedBS{\boldfaceGamma}$. Moreover, for every $\boldfaceGamma,\boldfaceGamma'$ s.t.\ $\boldfaceGamma(X)\subset \boldfaceGamma'(X)$ we have $\codedDS{\boldfaceGamma} \weireducible \codedDS{\boldfaceGamma'}$ and $\codedBS{\boldfaceGamma} \weireducible \codedBS{\boldfaceGamma'}$.

Notice also that the set of bad sequences through a $\boldfaceDelta^1_1$-quasi-order is $\boldfaceDelta^1_1$, hence it is straightforward to see that $\codedBS{\boldfaceDelta^1_1}\weireducible\codedChoice{\boldfaceSigma^1_1}{}{\Baire}\weiequiv\CBaire$. This shows also that $\codedBS{\boldfaceGamma}\weireducible \CBaire$ for every arithmetic $\boldfaceGamma$.

\begin{proposition}
    \thlabel{thm:gammaDS_cylinder}
	For every $\boldfaceGamma \in \{ \boldfaceSigma^0_k, \boldfacePi^0_k, \boldfaceDelta^0_k, \boldfaceSigma^1_1, \boldfacePi^1_1, \boldfaceDelta^1_1 \}$ the problems $\codedDS{\boldfaceGamma}$ and $\codedBS{\boldfaceGamma}$ are cylinders.
\end{proposition}
\begin{proof}
	The proof is a straightforward generalization of the proof of \thref{thm:DS_cylinder}. 
\end{proof}

\begin{theorem}
	\thlabel{thm:codedBS=jump_of_BS}
	For every $k\in\mathbb{N}$ and every $\boldfaceGamma\in \{ \boldfaceSigma, \boldfacePi, \boldfaceDelta\}$ 
	\begin{gather*}
		\codedDS{\boldfaceGamma^0_{k+1}} \weiequiv \codedDS{\boldfaceGamma^0_1}\compproduct \mflim^{[k]}\weiequiv \codedDS{\boldfaceGamma^0_1}^{(k)},\\ 
		\codedBS{\boldfaceGamma^0_{k+1}} \weiequiv \codedBS{\boldfaceGamma^0_1}\compproduct \mflim^{[k]}\weiequiv \codedBS{\boldfaceGamma^0_1}^{(k)}.
	\end{gather*}
\end{theorem}
\begin{proof}
	Fix $k$ and $\boldfaceGamma$ as above. The reduction $\codedDS{\boldfaceGamma^0_{k+1}}\weireducible \codedDS{\boldfaceGamma^0_1}\compproduct \mflim^{[k]}$ follows from the fact that 
	\[ \mflim^{[k]}\weiequiv \Comprehension{\boldfaceSigma^0_k}\weiequiv \Comprehension{\boldfacePi^0_k}\weiequiv \Comprehension{\boldfaceDelta^0_{k+1}},\] 
hence we can use $\mflim^{[k]}$ to compute a $\boldfaceGamma^0_1$-name for the input linear order, and then apply $\codedDS{\boldfaceGamma^0_1}$ to get a descending sequence. 
	
	Let us now prove the converse reduction. Since both $\mflim^{[k]}$ and $\codedDS{\boldfaceGamma^0_1}$ are cylinders, by the cylindrical decomposition there is an $e$ s.t.\ 
	\[ \codedDS{\boldfaceGamma^0_1}\compproduct \mflim^{[k]}\weiequiv \codedDS{\boldfaceGamma^0_1} \circ \Phi_e\circ \mflim^{[k]}. \]
	Given any $p \in \dom(\codedDS{\boldfaceGamma^0_1} \circ \Phi_e\circ \mflim^{[k]})$, the string $q:=\Phi_e(\mflim^{[k]}(p))$ is a $\boldfaceGamma^0_1$-name for a linear order $L_p$. Since $q$ is  $\Delta^{0,p}_{k+1}$, the condition $a\le_{L_p} b$ is  $\Gamma^{0,p}_{k+1}$ for every $a,b$.
	This shows that, given an input $p$ we can uniformly compute a $\boldfaceGamma^0_{k+1}$-name for the linear order $L_p$, and hence use $\codedDS{\boldfaceGamma^0_{k+1}}$ to compute an $<_{L_p}$-descending sequence.

	The equivalence $\codedDS{\boldfaceGamma^0_1}\compproduct \mflim^{[k]}\weiequiv \codedDS{\boldfaceGamma^0_1}^{(k)}$ follows from the fact that $\codedDS{\boldfaceGamma^0_1}$ is a cylinder.

	The same reasoning works, mutatis mutandis, to show that 
	\[ \codedBS{\boldfaceGamma^0_{k+1}} \weiequiv \codedBS{\boldfaceGamma^0_1}\compproduct \mflim^{[k]}\weiequiv \codedBS{\boldfaceGamma^0_1}^{(k)}. \qedhere\] 
\end{proof}

Using this theorem, the relativized version of \thref{thm:bs=ds} can be proved explicitly as follows:
\begin{corollary}
	\thlabel{thm:delta0kDS=delta0kBS}
For every $k\ge 1$, $\codedDS{\boldfaceDelta^0_k} \weiequiv \codedBS{\boldfaceDelta^0_k}.$
\end{corollary}
\begin{proof}
	See \ref{errata}.
\end{proof}


\subsection{\texorpdfstring{{\protect\boldmath $\Gamma$}$^0_k$\normalfont-$\mathsf{DS}$ and {\protect\boldmath $\Gamma$}$^0_k$\normalfont-$\mathsf{BS}$}{Gamma0k-DS and Gamma0k-BS}}
\label{sec:arithmetic_DS_BS}

We will now show that the hierarchy of $\boldfaceGamma$-$\DS$ problems does not collapse at any finite level. First we study the hierarchy of $\codedDS{\boldfaceDelta^0_k}$ problems by characterizing their first-order parts (\thref{thm:fop_Delta0kDS}). Then we prove the analogues of \thref{thm:below_DS_below_RT1k} and \thref{thm:below_DS_below_lim_k} for $\codedDS{\boldfaceDelta^0_k}$ (\thref{thm:below_Delta0kDS_below_other}).

For any sequence of problems $f_s \pmfunction{\textbf{X}_s}{\textbf{Y}_s}$, $s \in \mathbb{N}$, the \emph{countable coproduct} of the sequence is the problem $\bigsqcup_{s \in \mathbb{N}} f_s \pmfunction{\bigsqcup_s \textbf{X}_i}{\bigsqcup_s \textbf{Y}_i}$ defined by $\left(\bigsqcup_{s \in \mathbb{N}} f_s\right)(s,x) := \{s\} \times f_s(x)$. The problem $\bigsqcup_{s \in \mathbb{N}} f_s$ allows us access to exactly one $f_s$ of our choice.

\begin{theorem}
	\thlabel{thm:fop_Delta0kDS}
	For every $k \ge 1$,
	\[\firstOrderPart{\codedDS{\boldfaceDelta^0_k} } \weiequiv  \left(\bigsqcup_{s \in \mathbb{N}} \codedChoice{\boldfaceDelta^0_k}{}{s}\right) \compproduct \PiBound.\] 
\end{theorem}

We split the proof into two lemmas.

\begin{lemma}
	\thlabel{thm:f<=Delta0kDS_LO->f<=Delta0kCs_*_PiBound}
	For every $k\ge 1$, if $f\pmfunction{\mathbf{X}}{\mathbb{N}}$ and $f \weireducible \codedDS{\boldfaceDelta^0_k}$ then 
	\[ f \weireducible \left(\bigsqcup_{s \in \mathbb{N}} \codedChoice{\boldfaceDelta^0_k}{}{s}\right) \compproduct \PiBound.\] 
\end{lemma}
\begin{proof}
	Fix Turing functionals $\Phi$ and $\Psi$ which witness that $f \weireducible \codedDS{\boldfaceDelta^0_k}$. Given an $f$-instance with name $x$, $\Phi^x$ is a $\Delta^{0,x}_k$-code for the linear order $\le^x$. Consider the $\Sigma^{0,x}_k$ set
	\begin{align*}
		D := \{F \in \mathbb{N} \st {}&{} F \text{ codes a non-empty finite} <^x\text{-descending sequence and } \\
		&  \Psi^{x \oplus F} \text{ outputs some } j \in \mathbb{N}\}.
	\end{align*}
	
	We can uniformly express $D$ as the increasing union over $s \in \mathbb{N}$ of finite sets $D_s \subseteq \{0,\hdots,s\}$, which are uniformly $\Pi^{0,x}_{k-1}$.
	
	We now define the set
	\[ A:= \{ s \in\mathbb{N} \st (\forall F\in D_s)( F \notin \mathrm{Ext}_x) \}, \]
	where $\mathrm{Ext}_x$ is the set of finite sequences that extend to an infinite $<^x$-descending sequence. It is easy to see that $A$ is $\Pi^{1,x}_1$, as being extendible in a $\Delta^0_k$-linear order is a $\Sigma^1_1$ property.
	
	We show that $A$ is finite. Since $\le^x$ is a $\codedDS{\boldfaceDelta^0_k}$-instance, we can fix an infinite $<^x$-descending sequence $S$. By definition of Weihrauch reducibility, $\Psi^{x \oplus S}$ outputs some $f$-solution $j \in \mathbb{N}$. By the continuity of $\Psi$, there is some finite non-empty initial segment $F$ of $S$ such that $\Psi^{x \oplus F}$ outputs $j$. Hence for all sufficiently large $s$, we have $F \in D_s$.
	
	This shows that we can apply $\PiBound$ to $A$ to obtain some $b \in \mathbb{N}$ which bounds $A$. Note that $D_b$ must be nonempty. We now define the following non-empty subset of $D_b$:
	\[ B:=\left\{ F \in D_b \st (\forall G\in D_b)\left(\min_{<^x}(G) \le^x \min_{<^x}(F)\right) \right\}.\]
	Notice that all the quantifications are bounded. In particular, $B$ is a (non-empty) $\Delta^{0,x}_k$ subset of $D_b$ because $D_b$ is $\Pi^{0,x}_{k-1}$ and $\le^x$ is $\Delta^{0,x}_k$. Notice also that the definition of $B$ ensures that each of its elements is extendible (as we know that there is some extendible element in $D_b$). In particular, this shows that, for every $F\in B$, it is enough to run $\Psi^{x \oplus F}$ to compute an $f$-solution for the original instance. We can find such $F \in B$ by applying $\left(\bigsqcup_s \codedChoice{\boldfaceDelta^0_k}{}{s}\right)(b,B)$.
\end{proof}

Notice that $(\bigsqcup_s \codedChoice{\boldfaceDelta^0_1}{}{s})$ is computable, hence in case $k=1$ we obtain \thref{thm:bs=ds}.

\begin{lemma}
	\thlabel{thm:Delta0kCs_*_PiBound<=Delta0kDS_LO}
	For every $k\ge 1$,
	\[\left(\bigsqcup_{s \in \mathbb{N}} \codedChoice{\boldfaceDelta^0_k}{}{s}\right) \compproduct \PiBound\weireducible \codedDS{\boldfaceDelta^0_k}. \]
\end{lemma}
\begin{proof}
	Using the cylindrical decomposition we can write
	\[ \left(\bigsqcup_{s \in \mathbb{N}} \codedChoice{\boldfaceDelta^0_k}{}{s}\right) \compproduct \PiBound \weiequiv \left(\left(\bigsqcup_{s \in \mathbb{N}} \codedChoice{\boldfaceDelta^0_k}{}{s}\right) \times \id\right) \circ \Phi_e \circ (\PiBound\times \id{}) \]
	for some computable map $\Phi_e$. Let $\Phi_1,\Phi_2$ be computable maps s.t.\ $\Phi_e(p) = \coding{\Phi_1(p),\Phi_2(p)}$. Then we have
	\begin{align*}
		\left(\left(\bigsqcup_{s \in \mathbb{N}} \codedChoice{\boldfaceDelta^0_k}{}{s}\right) \times \id{}\right) \circ \Phi_e \circ (\PiBound\times \id{})(\coding{p_1,p_2}) = \\
		\coding{ \left(\bigsqcup_{s \in \mathbb{N}} \codedChoice{\boldfaceDelta^0_k}{}{s}\right)\Phi_1(\PiBound(p_1),p_2), \Phi_2(\PiBound(p_1),p_2)  }.
	\end{align*}
	Given an instance $\coding{p_1,p_2}$ of the above composition, we can think of $p_1$ as coding an input $A$ to $\PiBound$ via a tree $T$ s.t.\ for each $i$, $i\in A$ iff the subtree $T_i:=\{ \sigma \in T \st \sigma(0)=i \}$ of $T$ is well-founded. For any $b \in \PiBound(p_1)$, $\Phi_1(b,p_2)$ must be a name for an instance of $\bigsqcup_{s \in \mathbb{N}} \codedChoice{\boldfaceDelta^0_k}{}{s}$. Then $\pi_1\Phi_1(b,p_2)$ is a number $s$ and $\pi_2\Phi_1(b,p_2)$ is a $\boldfaceDelta^0_k$-name for a non-empty subset $A_s$ of $\{0,\hdots, s-1\}$, where $\pi_i(\coding{p_1,p_2})=p_i$ denotes the projection on the $i$-th component. Regardless of whether $b \in \PiBound(p_1)$, we will interpret $\pi_1\Phi_1(b,p_2)$ and $\pi_2\Phi_1(b,p_2)$ as above.
		
	We define a $\Delta^{0,\coding{p_1,p_2}}_k$ linear order as follows. First define
	\begin{align*}
		L:=\{ (\sigma, n) \st {}&{} \sigma\in p_1 \text{ and }  \\
			& \pi_1 \Phi_1(\sigma(0), p_2) \text{ outputs a number in less than } \abslength{\sigma} \text{ steps and} \\
			& n \text{ lies in the set named by } \pi_2 \Phi_1(\sigma(0), p_2) \}.
	\end{align*}
    We order the elements of $L$ by
	\[ (\sigma,n)\le_L (\tau,m) \defiff \sigma <_{\KB} \tau \lor (\sigma = \tau \land n \le m ).  \]
	It is easy to see that $(L,\le_L)$ is $\Delta^{0,\coding{p_1,p_2}}_k$. Notice that it is a linear order, as the pairs are ordered lexicographically where the first components are ordered according to the Kleene-Brouwer order on $\baire$ and the second components are ordered according to the order on $\mathbb{N}$. 

	Let $\sequence{q_i}{i\in\mathbb{N}}$ be an $<_L$-descending sequence, with $q_i=(\sigma_i,n_i)$. Notice that for each $i$ there is a $j>i$ s.t.\ $\sigma_j <_{\KB} \sigma_i$. Indeed, if there is an $i$ s.t.\ for all $j>i$ we have $\sigma_j = \sigma_i$ then, by definition of $\le_L$, the sequence $\sequence{n_j}{j>i}$ would be a descending sequence in the natural numbers, which is impossible.
	
	This implies that there is a subsequence $\sequence{q_{i_k}}{k\in\mathbb{N}}$ s.t.\ $\sequence{\sigma_{i_k}}{k\in\mathbb{N}}$ is a $<_{\KB}$-descending sequence. In particular, this implies that $T_{\sigma_0(0)}$ is ill-founded, i.e.\ $\sigma_0(0)\in \PiBound(p_1)$. Moreover, by definition of $L$, this implies that $n_0$ lies in the set named by $\pi_2 \Phi_1(\sigma_0(0), p_2)$.

	In other words, given an $<_L$-descending sequence $\sequence{q_i}{i\in\mathbb{N}}$ we have that $(\pi_1 q_0) (0) \in \PiBound(p_1)$ and $\pi_2 q_0 \in \left(\bigsqcup_{s \in \mathbb{N}} \codedChoice{\boldfaceDelta^0_k}{}{s}\right)\Phi_1(\PiBound(p_1),p_2)$. From this we can compute $\Phi_2(\pi_1 q_0, p_2)$ as well. This establishes the desired reduction.
\end{proof}

This completes the proof of \thref{thm:fop_Delta0kDS}.

With a small modification of the argument in the proof of \thref{thm:f<=Delta0kDS_LO->f<=Delta0kCs_*_PiBound} we can prove the following:

\begin{proposition}
	\thlabel{thm:below_Delta0kDS_below_other}
	Fix $k\ge 1$. For every $f\pmfunction{\mathbf{X}}{\mathbb{N}}$, 
	\[ f \weireducible \codedDS{\boldfaceDelta^0_k} \iff f \weireducible \PiBound \times \mflim^{[k-1]}.\]
	If, in particular, $f$ has codomain $N$ for some $N\ge 1$ then 
	\[ f \weireducible \codedDS{\boldfaceDelta^0_k} \iff f \weireducible \RT{1}{N}\compproduct \mflim^{[k-1]}.\]
	If, additionally, $f$ is single-valued, then 
	\[f \weireducible \codedDS{\boldfaceDelta^0_k} \iff  f \weireducible \mflim_N\compproduct\mflim^{[k-1]}. \]
\end{proposition}
\begin{proof}
	The right-to-left implication follows from \thref{thm:pibound<BS} and \thref{thm:codedBS=jump_of_BS}:
	\begin{align*}
		\PiBound \times \mflim^{[k-1]} & \weireducible \PiBound\compproduct\mflim^{[k-1]} \\ 
		& \weireducible \DS\compproduct\mflim^{[k-1]}\weiequiv\codedDS{\boldfaceDelta^0_k}.
	\end{align*}
	
	To prove the left-to-right implication, fix a pair of Turing functionals $\Phi$ and $\Psi$ witnessing the reduction $f \weireducible \codedDS{\boldfaceDelta^0_k}$. Fix an $f$-instance with name $x$ and let $\le^x$ be the $\Delta^{0,x}_k$ linear order defined by $\Phi^x$. 
	
	Define $D$, $D_s$ and $A$ as in the proof of \thref{thm:f<=Delta0kDS_LO->f<=Delta0kCs_*_PiBound}. In that proof, we applied $\PiBound$ to $A$ to obtain $b \in \mathbb{N}$. Then we restricted our attention to $B \subseteq D_b$. Here we will still apply $\PiBound$ to $A$, but we will concurrently consider a subset $B_s$ of each $D_s$. For each $s$, define
	\begin{gather*}
		B_s:=\left\{ F \in D_s \st (\forall G\in D_s)\left(\min_{<^x}(G) \le^x \min_{<^x}(F)\right) \right\},\\
		F_s:= \min B_s,
	\end{gather*}
	where $F_s$ is intended to be the empty sequence if $B_s$ (and hence $D_s$) is empty.

	Notice that $B_s \subset \{0,\hdots,s-1\}$ is $\boldfaceDelta^0_k$ (as each $D_s$ is $\boldfacePi^0_{k-1}$) and therefore $F_s$ is $\boldfaceDelta^0_k$. Since $\mflim^{[k-1]}\weiequiv \Comprehension{\boldfaceDelta^0_k}$, it can determine which $B_s$ is nonempty, and compute $F_s$ if $B_s$ is nonempty. Therefore the sequence $\sequence{F_s}{s\in\mathbb{N}}$ can be computed using $\mflim^{[k-1]}$. For every $b\in \PiBound(A)$ we have that $F_b$ is extendible to an infinite $<^x$-descending sequence and that $\Psi^{x \oplus F_s}$ converges to some $f$-solution $j$ (see also the proof of \thref{thm:f<=Delta0kDS_LO->f<=Delta0kCs_*_PiBound}).

	Assume now that $f$ has codomain $N$ for some $N\ge 1$. We can modify the above argument as follows: after computing the sequence $\sequence{F_s}{s\in\mathbb{N}}$, we consider the $\RT{1}{N}$-instance $c$ defined as 
	\[ c(s):= \begin{cases}
		0 & \text{if } F_s=\coding{},\\
		\Psi^{x \oplus F_s}(0) &\text{otherwise}.
	\end{cases}\]	
	Since $F_s$ is nonempty and extendible for cofinitely many $s$, if $c(s) = i$ for infinitely many $s$ (i.e., $c$ has an $\RT{1}{N}$-solution of color $i$), then there is an extendible $F_s$ s.t.\ $\Psi^{x \oplus F_s}(0)=i$, hence $i$ is an $f$-solution.
	
	If, additionally, $f$ is single-valued, then there is only one possible $i$ s.t.\ $c$ has a homogeneous solution with color $i$. This shows that the sequence $\sequence{c(s)}{s\in\mathbb{N}}$ has a limit, and therefore it suffices to use $\mflim_N$ to get the solution.

	The fact that $\RT{1}{N}\compproduct \mflim^{[k-1]}$ and $\mflim_N \compproduct \mflim^{[k-1]}$ are reducible to $\codedDS{\boldfaceDelta^0_k}$ follows from the fact that the compositional product is a degree theoretic operation, as $\RT{1}{N}\weireducible \DS$ (\thref{thm:below_DS_below_RT1k}), $\mflim_N\weireducible \DS$ (\thref{thm:below_DS_below_lim}) and $\codedDS{\boldfaceDelta^0_k}\weiequiv \DS\compproduct\mflim^{[k-1]}$ (\thref{thm:codedBS=jump_of_BS}).
\end{proof}

Notice that $\PiBound \times \mflim^{[k-1]}$ is not a first-order problem, so the first statement in \thref{thm:below_Delta0kDS_below_other} is not an alternative characterization of $\firstOrderPart{\codedDS{\boldfaceDelta^0_k}}$. It can be rephrased as 
\[ \firstOrderPart{\codedDS{\boldfaceDelta^0_k}} \weiequiv \firstOrderPart{(\PiBound \times \mflim^{[k-1]})}.\]

This concludes our discussion of the first-order problems that are Weihrauch reducible to $\codedDS{\boldfaceDelta^0_k}$. As for the deterministic part of $\codedDS{\boldfaceDelta^0_k}$:

\begin{corollary}
	\thlabel{thm:detpart_Delta0kDS}
	For every $k\ge 1$, $\DetPart{\codedDS{\boldfaceDelta^0_k}} \weiequiv \mflim^{[k]}$. 
\end{corollary}
\begin{proof}
	This follows from $\DetPart{\DS}\weiequiv\mflim$ (\thref{thm:below_DS_below_lim}) and the fact that, for cylinders, the jump commutes with the deterministic part (\thref{thm:detpart_commutes_jump}).
\end{proof}

\begin{theorem}
	\thlabel{thm:DS_hierarchy_theorem}
	For every $k\ge 1$,
	\[ \codedDS{\boldfaceDelta^0_k}\strictlyweireducible \codedDS{\boldfaceDelta^0_{k+1}}. \]
	In particular this shows that the $\boldfaceGamma$-$\DS$-hierarchy does not collapse at any finite level.
\end{theorem}
\begin{proof}
	This follows directly from \thref{thm:below_Delta0kDS_below_other} or, alternatively, from \thref{thm:detpart_Delta0kDS}. Indeed it suffices to notice that, for every $k\ge 1$, $\LPO^{(k)}\weireducible \mflim^{[k+1]}$ but $\LPO^{(k)}\not\weireducible \mflim^{[k]}$, as $\LPO^{(k)}$ is the characteristic function of a $\Sigma^0_{k+1}$-complete set while $\mflim^{[k]}$ is $\Sigma^0_{k+1}$-measurable.
\end{proof}

\begin{theorem}
	\thlabel{thm:Delta0_k+1-DS_LO=Pi0kDS_LO}
	For every $k\ge 1$, $\codedDS{\boldfaceDelta^0_{k+1}} \weiequiv \codedDS{\boldfacePi^0_k}$.
\end{theorem}
\begin{proof}
	The right-to-left reduction is trivial. To prove the left-to-right one it suffices to show that $\codedDS{\boldfaceDelta^0_1}' \weiequiv \codedDS{\boldfacePi^0_1}$ and the proof will follow from \thref{thm:codedBS=jump_of_BS} as 
	\[ \codedDS{\boldfaceDelta^0_{k+1}} \weiequiv \codedDS{\boldfaceDelta^0_1}' \compproduct \mflim^{[k-1]} \weiequiv \codedDS{\boldfacePi^0_1} \compproduct \mflim^{[k-1]} \weiequiv  \codedDS{\boldfacePi^0_k}.\]

	Let $p=\sequence{p_n}{n\in\mathbb{N}}$ be a sequence in $\Baire$ converging to the characteristic function of an ill-founded linear order $L$. In the following it is convenient to consider also the sequence $q=\sequence{q_n}{n\in\mathbb{N}}$, where $q_n(i):=p_n(\coding{i,i})$. Clearly $q$ converges to the characteristic function of $\dom(L)$ and is uniformly computable from $p$.

	For sake of readability, define the formula
	\[ \varphi(\sequence{x_n}{n\in\mathbb{N}}, \sigma) := (\forall i < \abslength{\sigma})( x_{\sigma(i)}(i) \neq x_{\sigma(i)+1}(i) \land (\forall j>\sigma(i))(x_j(i)=x_{j+1}(i)) ). \]
	Intuitively $\varphi$ says that, for each $i<\abslength{\sigma}$, $\sigma(i)$ codes the positions in which the sequence $\sequence{x_n}{n\in\mathbb{N}}$ changes for the last time in the $i$-th row. Let us also write $x_\sigma := \abslength{\sigma}-1$. 
	We define
	\begin{align*}
		M:=\{ (\sigma,\tau)\in\baire\times \baire \st {}& \varphi(q,\sigma) \land{} \\
			& q_{\sigma(x_\sigma)+1}(x_\sigma)=1 \land{} \\
			& \varphi(p,\tau) \land \abslength{\tau}= \coding{x_\sigma,x_\sigma}+1\}
	\end{align*}

	Notice that the first two conditions imply that $x_\sigma\in L$. Intuitively $x_{\sigma}$ is the $\le_\mathbb{N}$-largest element that is witnessed by $\sigma$ to enter in $L$. The last line says that $\tau$ is exactly as long as needed to witness all the relations between the elements of $L$ that are $\le_\mathbb{N} x_\sigma$.

	We order the set $M$ as follows:
	\[ (\sigma_0,\tau_0) \le_M (\sigma_1,\tau_1) \defiff x_{\sigma_0} \le_L x_{\sigma_1} \]
	Notice that $M$ is a $\Pi^{0,p}_1$ linear order as $M$ is $\Pi^{0,p}_1$ and the order $\le_M$ is $p$-computable: indeed, given two pairs $(\sigma_0,\tau_0),(\sigma_1,\tau_1)\in M$, we can use the longer string between $\tau_0$ and $\tau_1$ to $p$-compute whether $x_{\sigma_0} \le_L x_{\sigma_1}$. Notice also that, for each $l$, there is exactly one string $\sigma$ of length $l$ witnessing $\varphi(q,\sigma)$ (by minimality). The third line in the definition of $M$ implies that if $\sigma$ satisfies the first two conditions then there is a unique $\tau$ s.t.\ $(\sigma,\tau)\in M$. The linearity of $M$ follows by the linearity of $L$.

	To conclude the proof it is enough to notice that if $\sequence{(\sigma_i,\tau_i)}{i\in\mathbb{N}}$ is an $<_M$-descending sequence then $\sequence{x_{\sigma_i}}{i\in\mathbb{N}}$ is an $<_L$-descending sequence.
\end{proof}

The following is essentially a classical result (see e.g.\ \cite[Thm.\ 2.4]{Downey98}). The proof is simple enough that we can briefly sketch it.
\begin{theorem}
	\thlabel{thm:Sigma01_DS=DS}
	For every $k\ge 1$, $\codedDS{\boldfaceSigma^0_k} \weiequiv \codedDS{\boldfaceDelta^0_k}$.
\end{theorem}
\begin{proof}
	Given a $\boldfaceSigma^0_k$ linear order $L$, we can uniformly consider a sequence $\sequence{(L_s,\le_s)}{s\in\mathbb{N}}$ of $\boldfaceDelta^0_k$ linear orders approximating $L$. We then define
	\begin{gather*}
		M:=\{ (q,s) \st q\in L_s \text{ and } (\forall t<s)(q\notin L_t) \}, \\
		(p,s)\le_M (q,t) \defiff p\le_L q.
	\end{gather*}
	Notice that $(p,s)\le_M (q,t)$ can be written also as $p=q \lor (\forall i)(q\not\le_i p)$, hence $M$ is $\Delta^{0,L}_k$.
	Moreover, since for every $q\in L$ there is a unique $s$ s.t.\ $(q,s)\in M$, it is easy to see that $M$ is computably isomorphic to $L$. In particular, given an $<_M$-descending sequence we can obtain an $<_L$-descending sequence by projection.
\end{proof}

\begin{corollary}
	\thlabel{thm:recap_coded}
    For every $k \geq 1$, we have
\[ \codedDS{\boldfacePi^0_k} \weiequiv \codedDS{\boldfaceDelta^0_{k+1}} \weiequiv \codedDS{\boldfaceSigma^0_{k+1}}. \]
\end{corollary}
\begin{proof}
	The equivalence $\codedDS{\boldfacePi^0_k} \weiequiv \codedDS{\boldfaceDelta^0_{k+1}}$ was proved in \thref{thm:Delta0_k+1-DS_LO=Pi0kDS_LO}. The equivalence $\codedDS{\boldfaceDelta^0_{k+1}} \weiequiv \codedDS{\boldfaceSigma^0_{k+1}}$ was proved in \thref{thm:Sigma01_DS=DS}.
\end{proof}

\begin{theorem}
	\thlabel{thm:Sigma0kDS_LQO=Delta0_k+1DS_LO}
	For every $k\ge 1$, $\LPO^{(k)}\weireducible \codedBS{\boldfaceSigma^0_k}$ and therefore $\codedBS{\boldfaceSigma^0_k}\not\weireducible \codedDS{\boldfaceSigma^0_k}$.
\end{theorem}
\begin{proof}
    The second statement follows from the first because $\LPO^{(k)} \not\weireducible \codedDS{\boldfaceDelta^0_k}$ (proof of \thref{thm:DS_hierarchy_theorem}) and $\codedDS{\boldfaceDelta^0_k} \weiequiv \codedDS{\boldfaceSigma^0_k}$ (\thref{thm:Sigma01_DS=DS}).
    
	To prove the first statement, it is enough to show that $\LPO'\weireducible\codedBS{\boldfaceSigma^0_1}$, and the claim will follow by \thref{thm:codedBS=jump_of_BS} as 
	\[ \LPO^{(k)} \weireducible \LPO'\compproduct\mflim^{[k-1]} \weireducible \codedBS{\boldfaceSigma^0_1}\compproduct \mflim^{[k-1]}\weiequiv\codedBS{\boldfaceSigma^0_k}. \]

	Let $\sequence{p_s}{s\in\mathbb{N}}$ be a sequence in $\Baire$ converging to an instance $p$ of $\LPO$. For every $s\in\mathbb{N}$ we define (as we did in the proofs of \thref{thm:lpo'_DS_times_lim} and \thref{thm:lpo'<=findS})
	\[ g(s)= \begin{cases}
		i+1 & \text{if } i\le s \land p_s(i)\neq 0 \land (\forall j<i)(p_s(j)=0),\\
		0 & \text{otherwise.}
	\end{cases} \]
	
	Let us define a quasi-order $Q$ inductively: at stage $s=0$ we add $\coding{g(0),0}$. At stage $s+1$ we do the following:
	\begin{enumerate}
		\item if $g(s)=g(s+1)$ we put $\coding{g(s),s+1}$ immediately below $\coding{g(s),s}$;
		\item if $g(s)\neq g(s+1)$ we put $\coding{g(s+1),s+1}$ at the top and we put $\coding{-1,s+1}$ at the bottom. Moreover we collapse to a single equivalence class all the elements $\coding{g,t}$ with $t\le s$ and $g\neq -1$.
	\end{enumerate}
	This construction produces a quasi-order $(Q,\preceq_Q)$ which is computable in $\sequence{p_s}{s\in\mathbb{N}}$. 

	Notice that if there is an $s$ s.t.\ for every $t\ge s$, $g(t)=g(s)$ (in particular, this is the case if $\LPO(p)=1$) then the equivalence classes of $\preceq_Q$ form a linear order of type $n+\omega^*$ and every $\preceq_Q$-bad sequence is a descending sequence of the form $\sequence{\coding{g(s),s_n}}{n\in\mathbb{N}}$ for some strictly increasing sequence $\sequence{s_n}{n\in\mathbb{N}}$. On the other hand, if the sequence $\sequence{g(s)}{s\in\mathbb{N}}$ does not stabilize then the equivalence classes of $\preceq_Q$ are linearly ordered as $\omega^*$, where all the elements $\coding{g,s}$ with $g\neq -1$ are equivalent and lie in the top equivalence class. This shows that the construction produces a non-well quasi-order.
 
	For every $\preceq_Q$-bad sequence $\sequence{\coding{g_n,s_n}}{n\in\mathbb{N}}$ produced by $\codedBS{\boldfaceSigma^0_1}(Q)$, we compute the solution for $\LPO'(\sequence{p_s}{s\in\mathbb{N}})=\LPO(p)$ by returning $0$ if $g_1 \le 0$ and $1$ otherwise. We consider two cases. If the sequence $\sequence{g(s)}{s\in\mathbb{N}}$ stabilizes, then the sequence $\sequence{g_n}{n\in\mathbb{N}}$ is constant. Furthermore, its value is $0$ if $\LPO(p)=0$, otherwise its value is positive. On the other hand, if the sequence $\sequence{g(s)}{s\in\mathbb{N}}$ does not stabilize, then $\LPO(p) = 0$. Furthermore, for every $n>0$, we have $g_n=-1 \leq 0$. (The first element $\coding{g_0,s_0}$ may lie in the top equivalence class, in which case $g_0$ may be positive. Hence we check $g_1$ instead of $g_0$).
\end{proof}

\subsection{\texorpdfstring{{\protect\boldmath $\Gamma$}$^1_1$\normalfont-$\mathsf{DS}$ and {\protect\boldmath $\Gamma$}$^1_1$\normalfont-$\mathsf{BS}$}{Gamma11-DS and Gamma11-BS}}
\label{sec:analytic_DS_BS}

We now turn our attention to the analytic classes. Notice first of all that being a descending sequence through a $\boldfaceSigma^1_1$ linear order is a $\boldfaceSigma^1_1$-property, hence $\codedDS{\boldfaceSigma^1_1}\weireducible\codedChoice{\boldfaceSigma^1_1}{}{\Baire}\weiequiv \CBaire$. We will show that $\codedDS{\boldfaceSigma^1_1}$ is the strongest $\DS$-principle that is still reducible to $\CBaire$ (\thref{thm:Pi11CA<=Pi11DS}).
\begin{proposition}
	\thlabel{thm:Delta11-DS=DS*UCBaire}
	$\codedDS{\boldfaceDelta^1_1}\weiequiv \DS \compproduct \UCBaire$ and $\codedBS{\boldfaceDelta^1_1}\weiequiv \BS \compproduct \UCBaire$.
\end{proposition}
\begin{proof}
    We will only prove the first statement. The proof of the second statement is similar.
    
	To prove the left-to-right reduction, given a $\boldfaceDelta^1_1$ name for $L$ we use $\Comprehension{\boldfaceDelta^1_1}$ (which is known to be equivalent to $\UCBaire$, see \cite[Thm.\ 3.11]{KMP20}) to compute a $\boldfaceDelta^0_1$ name for $L$. We can then apply $\DS$ to find a descending sequence through $L$.

	To prove the converse reduction, using the cylindrical decomposition we can write 
	\[ \DS\compproduct\UCBaire \weiequiv \DS \circ \Phi_e \circ \UCBaire \]
	for some computable function $\Phi_e$.  In particular, given $T\subset \baire$ with a unique path $x$, $\Phi_e(x)$ is the characteristic function of a linear order $L$. Notice that $x$ is $\Delta^{1,T}_1$-computable. Indeed,
	\begin{align*}
		x(n)=k & \iff (\exists \sigma \in T)(\sigma\in \mathrm{Ext} \land \sigma(n)=k) \\
			& \iff (\forall \tau \in T)(\tau\in \mathrm{Ext} \rightarrow \tau(n)=k),
	\end{align*}
	where $\mathrm{Ext}$ is the set of finite strings that extend to a path through $T$ ($\sigma\in \mathrm{Ext}$ is a $\Sigma^{1,T}_1$ property). We can therefore obtain a $\Delta^{1,T}_1$ name for $L$ as 
	\[ a\le_L b \iff \Phi_e(x)(\coding{a,b})=1, \]
	and hence we use $\codedDS{\boldfaceDelta^1_1}$ to find a descending sequence through $L$.
\end{proof}
In particular, this implies that $\boldfaceDelta^1_1$ is the first level at which we can compute $\UCBaire$. Indeed, for every $k$, we showed in the proof of \thref{thm:DS_hierarchy_theorem} that $\LPO^{(k)}\not\weireducible\codedDS{\boldfaceDelta^0_k}$, while $\mflim^{[k]}\weireducible \UCBaire$ (see \cite[Sec.\ 6]{BdBPLow12}).

By adapting the proof of \thref{thm:delta0kDS=delta0kBS}, we can relativize \thref{thm:bs=ds} and obtain the following:
\begin{corollary}
	\thlabel{thm:delta11DS=delta11BS}
    $\codedDS{\boldfaceDelta^1_1} \weiequiv \codedBS{\boldfaceDelta^1_1}.$
\end{corollary}
\begin{proof}
	See \ref{errata}.
\end{proof}

Similarly, the proofs of \thref{thm:fop_Delta0kDS} and of \thref{thm:below_Delta0kDS_below_other} lead to the following equivalences:

\begin{theorem}
	\thlabel{thm:fop_Delta11-DS}
	\[ \firstOrderPart{\codedDS{\boldfaceDelta^1_1}} \weiequiv \firstOrderPart{(\PiBound \times \UCBaire)} \weiequiv \left(\bigsqcup_{s \in \mathbb{N}} \codedChoice{\boldfaceDelta^1_1}{}{s}\right) \compproduct \PiBound. \]
\end{theorem}

The deterministic part of $\codedDS{\boldfaceDelta^1_1}$ and $\codedDS{\boldfaceSigma^1_1}$ can be easily characterized using \thref{thm:Delta11-DS=DS*UCBaire}, as the following proposition shows.
\begin{proposition}
	\thlabel{thm:detpart_Delta11-DS}
	$\UCBaire\weiequiv \DetPart{\codedDS{\boldfaceDelta^1_1}}\weiequiv\DetPart{\codedDS{\boldfaceSigma^1_1}} $.
\end{proposition}
\begin{proof}
	The reductions $\UCBaire\weireducible \DetPart{\codedDS{\boldfaceDelta^1_1}}\weireducible\DetPart{\codedDS{\boldfaceSigma^1_1}}$ are straightforward from $\UCBaire\weireducible \codedDS{\boldfaceDelta^1_1}$ (\thref{thm:Delta11-DS=DS*UCBaire}), $\codedDS{\boldfaceDelta^1_1}\weireducible\codedDS{\boldfaceSigma^1_1}$ (trivial) and the fact that $\UCBaire$ is single-valued. To prove that $\DetPart{\codedDS{\boldfaceSigma^1_1}}\weireducible\UCBaire$ it is enough to notice that $\codedDS{\boldfaceSigma^1_1}\weireducible \CBaire$, and therefore $\DetPart{\codedDS{\boldfaceSigma^1_1}}\weireducible\DetPart{\CBaire}\weiequiv \UCBaire$ (\thref{thm:detpart_cbaire}).
\end{proof}

In particular, the deterministic part does not help us separate $\codedDS{\boldfaceDelta^1_1}$ and $\codedDS{\boldfaceSigma^1_1}$. Instead, we separate them by considering their first-order parts. We characterized $\firstOrderPart{\codedDS{\boldfaceDelta^1_1}}$ in \thref{thm:fop_Delta11-DS}. Notice that our proof (see the proof of \thref{thm:below_Delta0kDS_below_other}) cannot be extended to establish the same result for $\codedDS{\boldfaceSigma^1_1}$, because the definition of the corresponding $\sequence{F_s}{s\in\mathbb{N}}$ would not be $\boldfaceSigma^1_1$.

\begin{proposition}
	\thlabel{thm:par_Sigma11C_N<=Sigma11DS_LO}
	$\parallelization{\codedChoice{\boldfaceSigma^1_1}{}{\mathbb{N}}} \weireducible \codedDS{\boldfaceSigma^1_1}$.
\end{proposition}
\begin{proof}
	Let $\sequence{A_i}{i\in\mathbb{N}}$ be a sequence of non-empty $\boldfaceSigma^1_1$ subsets of $\mathbb{N}$. We define 
	\begin{gather*}
		L:=\{ (n, \sigma)\in \mathbb{N}\times \baire \st \abslength{\sigma}=n \land (\forall i<n)(\sigma(i)\in A_i) \},\\
		(n,\sigma) \le_L (m,\tau) \iff n> m \lor (n=m \land \sigma \le_{lex} \tau).
	\end{gather*}
	It is easy to see that $L$ is a $\boldfaceSigma^1_1$ linear order (the linearity follows from the linearity of $\le$ and of $\le_{lex}$).

	Let $\sequence{(n_i,\sigma_i)}{i\in\mathbb{N}}$ be an $<_L$-descending sequence. Notice that, since each $A_i\subset \mathbb{N}$, for each $n$ the set $\{ \sigma\in\baire \st (n,\sigma )\in L\}$ is $\le_{lex}$-well-founded. Therefore there must be a subsequence $\sequence{(n_{i_k},\sigma_{i_k})}{k\in\mathbb{N}}$ s.t.\ the sequence $\sequence{n_{i_k}}{k\in\mathbb{N}}$ is strictly increasing.

	This implies that, for each $n$, there is some $m$ s.t.\ $\abslength{\sigma_m}\ge n$. In particular, by definition of $L$, $(\forall i<n )(\sigma_m(i)\in A_i)$ and the claim follows.
\end{proof}

\thref{thm:par_Sigma11C_N<=Sigma11DS_LO} implies that $\codedChoice{\boldfaceSigma^1_1}{}{\mathbb{N}}\weireducible \firstOrderPart{\codedDS{\boldfaceSigma^1_1}}$. This, together with $\firstOrderPart{\CBaire}\weiequiv \codedChoice{\boldfaceSigma^1_1}{}{\mathbb{N}}$ (\thref{thm:fo_part_cbaire}) and the observation that $\codedDS{\boldfaceSigma^1_1}\weireducible \CBaire$, immediately yields the following:

\begin{corollary}
    $\firstOrderPart{\CBaire}\weiequiv\firstOrderPart{\codedDS{\boldfaceSigma^1_1}}\weiequiv \codedChoice{\boldfaceSigma^1_1}{}{\mathbb{N}}.$
\end{corollary}

As a consequence, $\CBaire$ and $\codedDS{\boldfaceSigma^1_1}$ cannot be separated by means of their first-order part. But $\codedDS{\boldfaceDelta^1_1}$ and $\codedDS{\boldfaceSigma^1_1}$ can, albeit somewhat indirectly:

\begin{proposition}
    \thlabel{thm:delta11-DS<sigma11-DS}
	$\codedDS{\boldfaceDelta^1_1}\strictlyweireducible \codedDS{\boldfaceSigma^1_1}$.
\end{proposition}
\begin{proof}
	Notice first of all that $\UCBaire\weireducible \parallelization{\PiBound}$. Indeed, given a tree $T\subset\baire$ with a unique path, we can consider the following sequence of $\Pi^{1,T}_1$ sets:
	\[A_n:= \{ k \in \mathbb{N} \st (\forall \sigma \in T)( (\exists x\in [T])(\sigma \pprefix x ) \rightarrow \sigma(n)\le k) \}.\]
	Clearly each $A_n$ is bounded by $x(n)$, where $x$ is the unique path through $T$. Given a solution $f\in \parallelization{\PiBound}(\sequence{A_n}{n\in\mathbb{N}})$, consider the space $X:=\{ \sigma\in \baire \st (\forall i<\length{\sigma})(\sigma(i)\le f(i)) \}$ and define $T_f:= T \cap X$. Notice that $[T_f]=[T]$. In particular, since $[X]$ is $f$-computably compact, we can uniformly (in $f$) compute the unique path through $[T_f]$ (see \cite[Thm.\ 7.23 and Cor.\ 7.26]{BGP17}). 

	If $\parallelization{\codedChoice{\boldfaceSigma^1_1}{}{\mathbb{N}}} \weireducible \codedDS{\boldfaceDelta^1_1}$ then, by \thref{thm:fop_Delta11-DS}, $\codedChoice{\boldfaceSigma^1_1}{}{\mathbb{N}} \weireducible\UCBaire \times \PiBound$ and therefore 
	\[\parallelization{\codedChoice{\boldfaceSigma^1_1}{}{\mathbb{N}}} \weireducible\parallelization{(\UCBaire \times \PiBound)} \weiequiv \parallelization{\PiBound}, \]
	contradicting $\parallelization{\codedChoice{\boldfaceSigma^1_1}{}{\mathbb{N}}} \not\weireducible  \parallelization{\PiBound}$ (\cite[Cor.\ 3.23]{KiharaADauriacChoice}).
\end{proof}

To separate $\codedDS{\boldfaceSigma^1_1}$ from $\CBaire$ we generalize a technique based on inseparable $\Pi^1_1$ sets, first used in \cite{KiharaADauriacChoice} to separate $\parallelization{\codedChoice{\boldfaceSigma^1_1}{}{\mathbb{N}}}$ from $\CBaire$. Consider the problem $\ATR_2\mfunction{\LO\times \Cantor}{\{0,1\}\times\Baire}$ defined in \cite[Def.\ 8.2]{gohatr}. It can be seen as a two-sided version of $\ATR$: it takes in input a pair $(L,A)$ and produces a pair $(i,Y)$ s.t.\ either $i=0$ and $Y$ is a $<_L$-infinite descending sequence or $i=1$ and $Y$ is a jump (pseudo)hierarchy starting from $A$. Jun Le Goh proved that $\UCBaire\strictlyweireducible \ATR_2 \strictlyweireducible \CBaire$ (\cite[Cor.\ 8.5 and 8.7]{gohatr}). 

Before proving the next theorem, we introduce the following notion of reducibility: for every $A,B \subset \Baire$, we say that $A$ is \textdef{Muchnik reducible} to $B$, and write $A\muchnikreducible B$ if, for every $b\in B$ there is a Turing functional $\Phi_e$ s.t.\ $\Phi_e(b)\in A$. Muchnik reducibility is the non-uniform version of Medvedev reducibility. For an extended presentation on these notions of reducibility see e.g.\ \cite{Simpson2008}.
\begin{theorem}
	\thlabel{thm:atr2_not<=_Sigma11DS_LO}
	$\ATR_2 \weiincomparable \codedDS{\boldfaceSigma^1_1}$, and therefore $\codedDS{\boldfaceSigma^1_1}\strictlyweireducible\CBaire$.
\end{theorem}
\begin{proof}
	The fact that $\codedDS{\boldfaceSigma^1_1} \not\weireducible \ATR_2$ follows from the fact that $\CBaire \weiequiv \mflim \compproduct\codedDS{\boldfaceSigma^1_1}$ while $\mflim\compproduct\ATR_2 \strictlyweireducible \CBaire$ (\cite[Cor.\ 8.5]{gohatr}).

	Let us now prove that $\ATR_2 \not\weireducible \codedDS{\boldfaceSigma^1_1}$.
	Assume towards a contradiction that there is a reduction witnessed by the maps $\Phi,\Psi$. Let $\sequence{L_e}{e\in\mathbb{N}}$ be an enumeration of the computable linear orders. Define the sets 
	\begin{gather*}
		S_e := \codedDS{\boldfaceSigma^1_1}(\Phi(L_e)),\\
		DS_e := \{ \sequence{x_n}{n} \in \Baire\st \sequence{x_n}{n} \text{ is an }<_{L_e}\text{-descending sequence} \},\\
		JH_e := \{ \sequence{y_n}{n} \in \Baire\st \sequence{y_n}{n} \text{ is a jump hierarchy on } L_e \}.
	\end{gather*}

	Notice that, for each $e$, $S_e$ is $\Sigma^1_1$ (being a descending sequence through a $\Sigma^1_1$ linear order is a $\Sigma^1_1$ condition) while $DS_e$ and $JH_e$ are arithmetic.

	Define now the sets 
	\begin{gather*}
		B:=\{ e\in\mathbb{N} \st DS_e \not\muchnikreducible S_e \},\\
		C:=\{ e\in\mathbb{N} \st JH_e \not\muchnikreducible S_e \},
	\end{gather*}
	where $\muchnikreducible$ represents Muchnik reducibility. In particular, if $X$ is (hyper)arithmetic and $Y$ is $\Sigma^1_1$ then $X\not\muchnikreducible Y$ is a $\Sigma^1_1$ condition, and therefore $B,C\in \Sigma^1_1(\mathbb{N})$.

	We now claim that $B\cap C=\emptyset$. Indeed, assume by contradiction that this is not the case and let $e\in B\cap C$. By definition of $B$ and $C$ this means that there are two descending sequences $\sequence{q_n}{n\in\mathbb{N}}$ and $\sequence{p_n}{n\in\mathbb{N}}$ in $\Phi(L_e)$ s.t.\ $\sequence{q_n}{n\in\mathbb{N}}$ does not compute any $<
	_{L_e}$-descending sequence and $\sequence{p_n}{n\in\mathbb{N}}$ does not compute any jump hierarchy on $L_e$. 

	In particular, if we run the backward functional $\Psi$ on $\sequence{q_n}{n\in\mathbb{N}}$ and $\sequence{p_n}{n\in\mathbb{N}}$ then, by continuity, there is an $n$ s.t.\ $\Psi(\sequence{q_i}{i<n})$ commits to producing a jump hierarchy on $L_e$ and $\Psi(\sequence{p_i}{i<n})$ commits to producing an $<_{L_e}$-descending sequence. Without loss of generality, assume that $q_n \le_{\Phi(L_e)} p_n$ (in the opposite case we just swap the roles of $\sequence{q_n}{n\in\mathbb{N}}$ and $\sequence{p_n}{n\in\mathbb{N}}$) and consider the sequence 
	\[ r:= \coding{p_0,\hdots, p_n,q_{n+1}, q_{n+2}, \hdots }. \]
	Notice that $\Psi(r)$ must produce an $<_{L_e}$-descending sequence, contradicting the fact that $\sequence{q_n}{n\in\mathbb{N}}$ does not compute any $<_{L_e}$-descending sequence.

	Let $\mathbf{wf_{LO}}$ be the set of indexes for the computable well-orderings and let $\mathbf{hds}$ be the set of indexes for computable linear orderings with a hyperarithmetic descending sequence. Notice that $\mathbf{wf_{LO}}\subset B$, because for each $e$ in $\mathbf{wf_{LO}}$, $DS_e = \emptyset\not\muchnikreducible A$ for every non-empty set $A$. Likewise, $\mathbf{hds}\subset C$, as any ill-founded linear order which has a hyperarithmetic descending sequence cannot support a jump hierarchy (see\footnotemark{} \cite[Thm.\ 4]{Friedman1976}).

	Since $B,C$ are disjoint and $\Sigma^1_1$, by $\Sigma^1_1$-separation there must be a $\Delta^1_1$ set separating them. Such a set would separate $\mathbf{wf_{LO}}$ and $\mathbf{hds}$ as well. This contradicts the fact that every $\Sigma^1_1$ set which separates $\mathbf{wf_{LO}}$ and $\mathbf{hds}$ must be $\Sigma^1_1$-complete \cite{GohPi11}.
\end{proof}\footnotetext{Friedman's result assumes that the linear order is adequate. We do not need this assumption because we choose to define jump hierarchies in a way such that each column (whether limit or successor) uniformly computes earlier columns, such as in \cite[Def.\ 3.1]{gohatr}. This allows us to run Friedman's proof without assuming adequacy.}

Finally we turn our attention to $\codedBS{\boldfaceSigma^1_1}$ and $\codedDS{\boldfacePi^1_1}$. We show below that these problems are much stronger in uniform computational strength than the problems considered so far. Indeed all the $\codedDS{\boldfaceGamma}$ problems, where $\boldfaceGamma = \boldfaceSigma^1_1$ or below, are s.t.\ 
\[ \codedDS{\boldfaceGamma}\strictlyweireducible \CBaire \weiequiv \mflim\compproduct\codedDS{\boldfaceGamma}. \]
In other words, $\codedDS{\boldfaceGamma}$ is arithmetically Weihrauch equivalent to $\CBaire$, which is prominent among the problems that are considered to be ``$\mathrm{ATR}_0$ analogues in the Weihrauch lattice'' \cite{KMP20}.

On the other hand, a natural analogue of $\boldfacePi^1_1\mathrm{-CA}_0$ in the Weihrauch lattice is $\PiCA$, which can be phrased as ``given a sequence $\sequence{T_n}{n\in\mathbb{N}}$ of trees in $\baire$, produce $x\in \Cantor$ s.t., for every $n$, $x(n)=1$ iff $[T_n]=\emptyset$''. 

We can notice that, using \cite[Thm.\ 6.5]{MarconeNWT96}, $\PiCA$ is equivalent to the problem of finding the leftmost path through an ill-founded tree. Using this fact we show that $\codedBS{\boldfaceSigma^1_1}$ and $\codedDS{\boldfacePi^1_1}$ are in the realm of $\boldfacePi^1_1\mathrm{-CA}_0$.
\begin{theorem}
	\thlabel{thm:Pi11ca<Sigma11DS_LQO}
	$\PiCA\weireducible\codedBS{\boldfaceSigma^1_1}$.
\end{theorem}
\begin{proof}
	Let $T\subset\baire$ be an ill-founded tree. For each $\sigma \in T$, let $T_\sigma:=\{\tau \in T \st \tau\prefix \sigma \lor \sigma\prefix \tau \}$. We define a quasi-order on the extendible strings in $T$:
	\begin{gather*}
		Q:=\{ \sigma \in T\st [T_\sigma]\neq \emptyset\},\\
		\sigma\preceq_Q \tau \defiff (\exists \rho \in Q)(\rho <_{lex} \sigma) \lor \tau \prefix \sigma.
	\end{gather*}

	It is easy to see that $(Q,\preceq_Q)$ is $\Sigma^{1,T}_1$. Moreover, all the $\sigma$ which are not prefixes of the leftmost path collapse in a bottom equivalence class. This shows that the equivalence classes of $Q$ are linearly ordered as $1+\omega^*$. To conclude the proof it is enough to notice that any $<_Q$-descending sequence gives longer and longer prefixes of the leftmost path, hence it computes $\PiCA$.
\end{proof}

\begin{corollary}
    $\codedDS{\boldfaceSigma^1_1} \strictlyweireducible \codedBS{\boldfaceSigma^1_1}$.
    \thlabel{thm:Pi11DS<Pi11BS}
\end{corollary}
\begin{proof}
    We have $\codedDS{\boldfaceSigma^1_1} \weireducible \CBaire \strictlyweireducible \PiCA \weireducible \codedBS{\boldfaceSigma^1_1}$.
\end{proof}

\begin{theorem}
	\thlabel{thm:Pi11CA<=Pi11DS}
	$\PiCA\weireducible \codedDS{\boldfacePi^1_1}$.
\end{theorem}
\begin{proof}
	Let $T\subset \baire$ be an ill-founded tree. For each $\sigma \in T$, let $T_\sigma:=\{\tau \in T \st \tau\prefix \sigma \lor \sigma\prefix \tau \}$. We define a linear order
	\begin{gather*}
		L:= \{ \sigma\in T \st (\forall \tau \le_{lex} \sigma)([T_{\tau}]=\emptyset \lor \tau \prefix \sigma) \},\\
		\le_L := \le_{\KB(T)}.
	\end{gather*}
	Clearly $(L,\le_L)$ is a $\Pi^{1,T}_1$ linear order. Notice that if $\sigma\in L$ and $[T_\sigma]\neq\emptyset$ then $\sigma$ must be a prefix of the leftmost path. Moreover if $\rho$ is strictly lexicographically above the leftmost path then $\rho \notin L$. In other words, $L$ is the subset of $T$ that is lexicographically below the leftmost path.

	Moreover, every string that is not a prefix of the leftmost path lies in the well-founded part of $L$ (by definition of $\KB$). In particular every $<_L$-descending sequence computes arbitrarily long prefixes of the leftmost path.
\end{proof}

\section{Conclusions}
In this paper we explored the uniform computational content of the problem $\DS$, and showed how it lies ``on the side" w.r.t.\ the part of the Weihrauch lattice explored so far. We now draw the attention to some of the questions that did not receive an answer.

The problem $\KL$ is the multi-valued function corresponding to K\"onig's lemma, and it can be phrased as ``find a path through an infinite finitely-branching tree". It is known that $\KL \weiequiv \CCantor' \weiequiv \parallelization{\RT{1}{2}}$. 
\begin{question}
    $\KL \weireducible \DS$?
\end{question}
We know that, if such a reduction exists, it must be strict (as $\KL$ is an arithmetic problem). On the other hand, none of the characterizations we used in  Section~\ref{sec:ds} to describe the lower cone of $\DS$ can be used to prove a separation.

In Section~\ref{sec:detpar_literature} we introduced the problem $\wList_{\Cantor,\leq\omega}$. Similarly to $\DS$, this problem does not fit well within the effective Baire hierarchy: $\DetPart{\wList_{\Cantor,\leq\omega}}\weiequiv \mflim$, but $\wList_{\Cantor,\leq\omega}^{[3]} \weiequiv \UCBaire$ (\cite[Prop.\ 6.14 and Cor.\ 6.16]{KMP20}), hence in particular $\wList_{\Cantor,\leq\omega}$ is not arithmetic.
\begin{question}
    $\wList_{\Cantor,\leq\omega} \weireducible \DS$?
\end{question}
Our results imply that $\DS \not\weireducible \wList_{\Cantor,\leq\omega}$ (as $\DS\compproduct\DS \weiequiv \CBaire$), and hence a reduction would be strict.

In the context of $\codedDS{\boldfaceGamma}$, there are a few problems that resisted full characterization. In particular:

\begin{question}
    $\codedDS{\boldfaceDelta^0_2} \weireducible \codedBS{\boldfaceSigma^0_1}$?
\end{question}
We expect that an answer to this question will yield a solution for every $k$ (by relativization).

We notice that, in the statements involving $\codedBS{\boldfaceGamma}$ we proved slightly more than what claimed: indeed, in all the reductions, the quasi-order built is a \textdef{linear quasi-order}, i.e.\ a quasi-order whose equivalence classes are linearly ordered. Notice that every bad sequence through a non-well linear quasi-order is actually a descending sequence. If we introduce the problem $\codedDS{\boldfaceGamma}_{LQO}$ by restricting $\codedBS{\boldfaceGamma}$ to linear quasi-orders, our results imply that 
\[ \codedDS{\boldfaceDelta^0_k}\strictlyweireducible \codedDS{\boldfaceSigma^0_{k+1}}_{LQO} \weireducible \codedBS{\boldfaceSigma^0_{k+1}}. \]
A natural question is therefore
\begin{question}
    $\codedBS{\boldfaceSigma^0_{k+1}} \weireducible \codedDS{\boldfaceSigma^0_{k+1}}_{LQO}$?
\end{question}
A negative answer would imply that the possibility of having infinite antichains provides extra uniform strength.

A very important structure that is left out of the picture is the one of partial orders. In the same spirit of the paper we can consider the problems $\codedDS{\boldfaceGamma}_{PO}$ and $\codedBS{\boldfaceGamma}_{PO}$. The former is readily seen to be equivalent to $\CBaire$ (see also the comment before \thref{def:bs}). 
\begin{question}
    What is the relation between $\codedBS{\boldfaceSigma^0_1}_{PO}$ and the problems $\DS\weiequiv \codedDS{\boldfaceSigma^0_1}$, $\codedDS{\boldfaceSigma^0_1}_{LQO}$ and $\codedBS{\boldfaceSigma^0_1}$?
\end{question}
Answering these questions would yield very interesting insights on how the possibility to have equivalent non-equal elements can enhance the uniform computational strength.

\bibliographystyle{amsplain}
\bibliography{bibliography}

\end{document}